\documentclass[11pt]{amsart}
\usepackage[english]{babel}
\usepackage[utf8]{inputenc}

\usepackage{geometry}

\usepackage[dvipsnames]{xcolor}

\usepackage{graphicx}
\usepackage{caption}
\usepackage{subcaption}

\usepackage{amsfonts}
\usepackage{bm}

\usepackage[ruled, vlined, linesnumbered]{algorithm2e}
\SetKwInput{KwInput}{Input}                
\SetKwInput{KwOutput}{Output}              

\usepackage[normalem]{ulem}

\usepackage{pgf,tikz,pgfplots}
\pgfplotsset{compat=1.15}

\usepackage{forest}

\usepackage{chngpage}

\usepackage{mathrsfs}
\usetikzlibrary{arrows}

\graphicspath{{img/}}

\usepackage{mathtools}
\usepackage{hyperref}

\DeclarePairedDelimiter\abs{\lvert}{\rvert}
\DeclarePairedDelimiter\norm{\lVert}{\rVert}

\usepackage{amsthm}
\newtheorem{theorem}{Theorem}[section]

\newtheorem{proposition}[theorem]{Proposition}

\theoremstyle{definition}
\newtheorem{definition}{Definition}[section]
\newtheorem{remark}{Remark}

\title[Euler Characteristic Curves and Profiles]{Euler Characteristic Curves and Profiles: a stable shape invariant for big data problems}
\author{Pawe{\l} D{\l}otko}
\thanks{PD and DG acknowledges support by Dioscuri program initiated by the Max Planck Society, jointly managed with the National Science Centre (Poland), and mutually funded by the Polish Ministry of Science and Higher Education and the German Federal Ministry of Education and Research}
\address{Dioscuri Centre in Topological Data Analysis\\
Mathematical Institute, Polish Academy of Sciences\\
Warsaw, Poland}
\email{pdlotko@impan.pl}

\author{Davide Gurnari}
\address{Dioscuri Centre in Topological Data Analysis\\
Mathematical Institute, Polish Academy of Sciences\\
Warsaw, Poland}
\email{dgurnari@impan.pl}

\begin{document}

\begin{abstract}
    Tools of Topological Data Analysis provide stable summaries encapsulating the shape of the considered data. Persistent homology, the most standard and well studied data summary, suffers a number of limitations; its computations are hard to distribute, it is hard to generalize to multifiltrations and is computationally prohibitive for big data-sets. 
    In this paper we study the concept of Euler Characteristics Curves, for one parameter filtrations and Euler Characteristic Profiles, for multi-parameter filtrations. While being a weaker invariant in one dimension, we show that Euler Characteristic based approaches do not possess some handicaps of persistent homology; we show efficient algorithms to compute them in a distributed way, their generalization to multifiltrations and practical applicability for big data problems. In addition we show that the Euler Curves and Profiles enjoys certain type of stability which makes them robust tool in data analysis. Lastly, to show their practical applicability, multiple use-cases are considered. 
\end{abstract}

\maketitle

\section{Introduction}

Topological Data Analysis since its beginning \cite{edelsbrunner_topological_2002}, \cite{singh_topological_2007} has brought attention in the data science community. Topological tools, like persistent homology \cite{edelsbrunner_computational_2022} and mapper \cite{singh_topological_2007} were used in multiple tasks in material science \cite{lee_quantifying_2017}, \cite{dlotko_topological_2016}, \cite{hiraoka_hierarchical_2016}, medicine \cite{nicolau_topology_2011} and many more. In time, persistent homology has been successfully integrated with machine learning pipelines, and mapper became an exploratory data analysis tool. In this work we will extend on the path of persistent homology. With its successes, attempts were made to apply it in task of big data analysis. 
However, the progress is minimal. While there exist a single distributed implementation \cite{bauer_distributed_2013}, it does not scale up and was not extensively used in big data analysis. In practice, mostly various sequential implementations are used \cite{gudhi:urm}. To bypass the problem of too large input, a number of sparsification techniques \cite{silva_topological_2004}, \cite{sheehy_linear-size_2013} as well as bootstrap \cite{chazal_bootstrap_2013} and zig-zag \cite{carlsson_zigzag_2010} approaches were proposed. While they scale up to problems of a certain size, they tend to bypass the big data challenge rather than proposing a solution for it.   

In this paper we extend the tool of classical Euler characteristic and Euler characteristic curves. The new contributions include: 
\begin{itemize}
    \item A proof of stability of the Euler Characteristic Curve (ECC)  with respect to the 1-Wasserstein distance between persistence diagrams;
    \item A generalization of the Euler Characteristic Curve to the multiparamenter filtration case, with arbitrary number of parameters, that we denote as Euler Characteristic Profile (ECP);
    \item An analysis of the stability of such ECPs;
    \item Distributed algorithms to compute the exact ECC for Vietoris-Rips and cubical complexes that can be naturally extended to the multiparameter case. An Python implementation of such algorithms is provided as scikit-learn \cite{scikit-learn} compatible package.
    \item Discussion of methods to compare and vectorize ECCs and ECPs;
    \item Examples of applications of the ECC/ECP to real word data.
\end{itemize}

\label{sec:related_work}
While we are not aware of any distributed algorithm to compute Euler Characteristic Curves of a Vietoris-Rips complex, Heiss and Wagner \cite{heiss_streaming_2017} describe a streaming algorithm to compute the ECC from cubical complexes which has also been adapted for GPU computations \cite{wang_gpu_2022}. While their implementation is very fast we see no straightforward way to generalize it to the multiparameter filtration case.
To the best of our knowledge the concept of Euler Characteristic Profiles of arbitrary dimension is novel in the literature. There are however some works that focus on the bifiltration case, known as Euler Characteristic Surfaces. It was used in an applied setting by Roy et al. \cite{roy_understanding_2020} to analyze drying droplets but no topological background is provided. Beltramo et al. \cite{beltramo_euler_2021} gave a description of Euler Characteristic Surfaces in the persistence homology framework and apply it to obtain a descriptor of both pointcloud and image based data. Moreover, they provide a Python implementation of their algorithms which however requires the input bifiltration to be binned. Chen et al. \cite{chen2022tampsgcnets} introduced a time-aware multipersistence Euler-Poincaré surface to describe dynamical networks and proved its weak $L_1$ stability. A recent preprint by Perez \cite{perez_euler_2022} analyzes the stability of Euler and Betti curves of stochastic processes  on compact Riemannian manifolds.

\section{Euler characteristic curves (and profiles)}

In this section we introduce the essential mathematical concepts needed to define Euler Characteristic Curves and Profiles. For an exhaustive presentation we refer to classic textbooks like \cite{hatcher_algebraic_2002} and \cite{edelsbrunner_computational_2022}. 

\begin{definition}
A \emph{CW} or \emph{cell complex} $X$ is a topological space that can be built up starting from a discrete set $X^0$ of  $0$-dimensional cells and then inductively creating the $n$-skeleton $X^n$ by attaching $n$-cells to $X^{n-1}$ along their boundary. The process can be stopped at some finite dimension or can continue indefinitely. A subset $A \subseteq X$ is a \emph{subcomplex} of $X$ if with each cell of $A$, all its lower dimensional cells enters $A$.
\end{definition}

\begin{remark}
    Since we are interested in applying this machinery to analyze real word data we will always assume that our complexes are finite.
\end{remark}
 While the theory can be built in the general CW complex setting, the algorithms we present in Section \ref{sec:Algo} are specific to two different specializations that are used to represent different types of data: simplicial and cubical complexes.

\begin{definition}
\label{def:simplicialcomplex}
 An \emph{abstract simplicial complex} is a finite collection of sets $K$ such that $\sigma \in K$ and $\tau \subseteq \sigma$ implies $\tau \in K$. The sets in $K$ are called \emph{simplicies} and the \emph{dimension} of a simplex is $dim (\sigma) = card(\sigma)-1$. We will often refer to 0-simplices as \emph{vertices}, and to 1-simplices as \emph{edges}. Given a simplex $s = \{v_0,\ldots,v_k\}$, its \emph{boundary} is $\partial s = \sum_{i=0}^{k} (-1)^i \{v_0,\ldots,\hat{v_i},\ldots,v_k\}$, where $\hat{v_i}$ denotes that the vertex $v_i$ is removed from the simplex. Simplices $\{v_0,\ldots,\hat{v_i},\ldots,v_k\}_{i=0}^k$ are in the boundary of $s$.
\end{definition}

There are different ways of obtaining an abstract simplicial complex from point cloud data such as the Čech, the Vietoris-Rips and the Alpha constructions \cite{edelsbrunner_computational_2022}, in Section \ref{sec:VR} we describe the Vietoris-Rips construction.

\begin{definition}
    An \emph{elementary interval} is a subset of $\mathbb{R}$ of the type $I = [l, l+1]$ or $I=[l, l]$, for some integer $l$. The first type is called non-degenerate interval while the second is a degenerate interval. An \emph{elementary cube} $C$ is a product of elementary intervals $C= I_1 \times \cdots \times I_n$ and its dimension is the number of non-degenerate intervals in the product. 
    The \emph{boundary} of an elementary interval is $\partial[l, l+1] = [l+1, l+1] + [l, l]$ and $\partial[l, l] = 0$. The boundary of an elementary cube is then defined as $\partial C = \partial(I_1 \times \cdots \times I_n) = \sum_{i=1}^n I_1 \times \cdots \times \partial I_i \times\cdots \times I_n$.
Similarly to the simplicial complex case, a \emph{cubical complex} $K$ is a collection of elementary cubes closed under operation of taking boundary
\end{definition}

One of most common use case of cubical complexes is image data. In Section \ref{sec:cube} we describe how to build a filtered cubical complex from an $n$-dimensional image, by identifying the image's pixels with top dimensional cells.


In what follows we will refer to simplices and cubes, as elements of a simplicial or a cubical complex jointly as \emph{cells} in a \emph{cell complex}. A cell $\tau$ is said to be a \emph{face} of $\sigma$ if $\tau$ is in the boundary of $\sigma$.

\begin{definition}
\label{def:chain}
Let $K$ be a cell complex and $d$ a dimension. A \emph{d-chain} is a formal sum of $d$-cells in $K$, namely $c = \sum a_i \sigma_i$ where the $\sigma_i$ are the $d$-cells and the $a_i$ are the coefficients.  
\end{definition}

There are many possible choices for the group of coefficients. A standard approach in computational topology is to use \emph{modulo 2 coefficients}, i.e. the $a_i$ can be either $0$ or $1$ and satisfy $1+1=0$ \footnote{Using modulo 2 coefficients allows us to get rid of the $(-1)^i$ in the definition of the boundary in \ref{def:simplicialcomplex}}. Other options include integer, rational or real coefficients. 

Two $d$-chains can be added component-wise. Namely, given $c = \sum a_i \sigma_i$ and $c' = \sum b_i \sigma_i$, $c+c'= \sum (a_i+b_i) \sigma_i$. Therefore, we can define the \emph{group of d-chains} $\bm{C}_d = \bm{C}_d(K)$. The boundary of a $d$-chain is the sum of the boundaries of its cells $\partial c = \sum a_i \partial \sigma_i$ , which is a $(d-1)$-chain. Since the boundary commutes with the addition operation, we can define -for each dimension $d$- the \emph{boundary homomorphism} $\partial_d : \bm{C}_d \rightarrow \bm{C}_{d-1}$.

A \emph{$d$-cycle} is a $d$-chain with empty boundary $\partial c = 0$. A \emph{$d$-boundary} is a $d$-chain which is the boundary of a $(d+1)$-chain. Since $\partial$ commutes with addition, we have the \emph{group of $d$-cycles} $\bm{Z_d} = \bm{Z}_d(K)$ , and the \emph{group of $d$-boundaries} $\bm{B_d} = \bm{B}_d(K)$. It is a fundamental result that $\partial_d \partial_{d+1} c = 0$ for every dimension $d$ and every $(d+1)$-chain $c$. This means that the boundary of a boundary is always zero, in other words $\bm{B}_d$ is a subgroup of $\bm{Z}_d$. This leads to the following definition.

\begin{definition}
\label{def:homology}

    The \emph{d-th homolgy group} is the $d$-th cycle group modulo the $d$-th boundary group, $\bm{H}_d = \bm{Z}_d / \bm{B}_d$. The \emph{d-th Betti number} is the rank of this group, $\beta_d = \text{rank}(\bm{H}_d)$.

\end{definition}

\begin{definition}
\label{def:filtration}
    Let $K$ be a cell complex. A \emph{filtration} of $K$ is a sequence of nested subcomplexes $\emptyset = K_0 \subseteq K_1 \subseteq \cdots \subseteq K_n = K$. Such sequence is finite for finite complexes. It can be obtained by means of a \emph{filtration function} over $K$, a monotonic non-decreasing function $f: K \rightarrow \mathbb{R}$ such that $f(\tau) \leq f(\sigma)$ if $\tau$ is a face of $\sigma$.
    Note that every \emph{sublevel set} $K_t = f^{-1}\left(-\infty, t\right]$ is a subcomplex of $K$ for every $t \in \mathbb{R}$. 
\end{definition}


 For each dimension $d$, such a filtration corresponds to a sequence of homology groups $0 = \bm{H}_d(K_0) \rightarrow \bm{H}_d(K_1) \rightarrow \cdots \rightarrow \bm{H}_d(K_n) = \bm{H}_d(K)$. For every $i < j$, the homomorphism $f_d^{i,j} : \bm{H}_d(K_i) \rightarrow \bm{H}_d(K_j)$ is induced from the inclusion map of $K_i$ into $K_j$.

\begin{definition}
    The \emph{$d$-th persistent homology groups} are the images of the homomorphisms $\bm{H}_d^{i,j} = \text{im}f_d^{i,j}$. The ranks of these groups are the \emph{$d$-th persistent Betti numbers} $\beta_d^{i,j} = \text{rank}(\bm{H}_d^{i,j})$. 
\end{definition}

Intuitively, the $d$-th persistent Betti number $\beta_d^{i,j}$ counts how may homology classes of $K_i$ are still present in $K_j$. There are two scenarios in which a homology class from $K_i$ may not be present in $K_j$ - it may either became trivial, or it may became identical (homologous) to a class that was created earlier.

\begin{definition}
    The $k$-th dimensional \emph{persistence diagram} of a filtered complex $K$, $Dgm_k(K)$ is a multiset of points in the extended real plane $(\mathbb{R} \cup \{ \infty \}) \times (\mathbb{R} \cup \{ \infty \})$. The multiplicity of each point $(b, d)$ indicates the number of independent $k$-dimensional classes that are born at filtration value $b$ and die at filtration value $d$.
\end{definition}

All the points on the diagonal are always included, with countable multiplicity, in a persistence diagram, in order to make sense of the following.
\begin{definition}
    A \emph{matching} of two persistence diagrams $C$ and $D$ is a bijection $\eta: C \rightarrow D$ possibly to or from points on the diagonal.
\end{definition}

\begin{definition}
    The 1-Wasserstein distance between two $k-$dimensional persistence diagrams $C, D$ is 
    \begin{equation*}
        W_1(C, D) = \inf_{\eta:C \rightarrow D} \sum_{(b,d) \in C} \norm{(b,d) - \eta(b, d)}_\infty^1
    \end{equation*}
    where $\eta$ is a matching of $C$ and $D$.
\end{definition}

\begin{definition}
The \emph{Euler Characteristic} of a cell complex $K$ is the alternating sum of the number of its cells in each dimension
\[
\chi(K) = \sum_d (-1)^d |K^d| \quad .
\]
\end{definition} 
Where $K^d$ denotes the $d$ dimensional cells in $K$. Thanks to the Euler-Poincaré formula, the Euler characteristic can also be expressed as the alternating sum of the Betti numbers, the ranks of the cell complex's homology groups: $\chi(K) = \sum_d (-1)^d \beta_d(K)$ \cite{edelsbrunner_computational_2022}.

\begin{definition}
Let us consider a filtered complex $K$ with filtration function $f: K \rightarrow \mathbb{R}$. We can define its \emph{Euler Characteristic Curve} as a function that assign an Euler number $\chi$ for each filtration level $t \in \mathbb{R}$
\[
ECC(K, t) = \chi(K_t) \quad .
\]
Recall that $K_t = f^{-1}\left(-\infty, t\right]$ is a subcomplex of $K$ for every $t \in \mathbb{R}$.
\end{definition}
\label{def:ECC}



We are now interested in extending the concept of Euler Characteristic Curve to the more general Multidimensional Persistence setting~\cite{carlsson_theory_2007}. In order to do so, we need to generalize Definition~\ref{def:filtration} to families of nested complexes indexed by posets. While Multidimensional Persistence is a vibrant and active research topic, in this paper we will only make use of the basic concepts. We refer the interested reader to \cite{botnan_introduction_2022} for a modern introduction to the topic.

\begin{definition}
\label{def:multifiltration}

Let $K$ be a cell complex and $P$ a poset. A \emph{P-indexed filtration} on $K$ is a family of nested complexes such that $K_x$ is a subcomplex of $K$ for each $x \in P$, and $K_x \subseteq K_y$ whenever $x \leq y$. 
If $P = T_1 \times \cdots \times T_n$ where each $T_i$ is a totally ordered set, we call  a \emph{multiparameter} or \emph{n-parameter filtration}. 

\end{definition}

It is a natural question to ask whether the idea of sublevel sets of a filtration function could be extended too. In general, this is not the case. It can be achieved only when each cell of $K$ first appears in the filtration at some unique minimal index in $P$.
\begin{definition}
    Let $K$ be a cell complex, $P$ a poset and $f$ a function $f: K \rightarrow P$. The $sublevel filtration$ of $f$ is a family of complexes of the type
    \[
    K_p = \{\sigma \in K \, | \, f(\sigma) \leq p \} \, .
    \]
    A filtration isomorphic to a sublevel filtration is said to be \emph{1-critical}. A filtration that is not 1-critical is said to be \emph{multicritical}.
\end{definition}

\begin{definition}
\label{def:ECP}
The \emph{Euler Characteristic Profile} (ECP) of a $P$-filtered complex $K$ is a function that assign to any value $p \in P$ the Euler characteristic of the corresponding subcomplex $K_p$.
\[
ECP(K, p) = \chi(K_p) \quad .
\]
\end{definition}

For the rest of the paper we will focus on the case $P=\mathbb{R}^n$.

\begin{remark}
    The two dimensional ECP already appeared in the literature and it is known as Euler Characteristic Surface \cite{roy_understanding_2020}, \cite{beltramo_euler_2021}, \cite{chen2022tampsgcnets}. It was however defined only for the Cartesian product of two one-parameter filtrations and it is treated as matrix in the following way. Given a bi-filtering function $F: K \rightarrow  \mathbb{R}^2 $ over $K$ and a set of threshold values $I = \{(a_i, b_j) \, | \, 1 \leq i \leq m , 1 \leq j \leq n\}$, the Euler Characteristic Surface is the $m\times n$ integer valued matrix $S$ whose entries are $S_{ij} = \chi(K_{ij}) = \chi(F^{-1} \left( (-\infty, a_i] \times (-\infty, b_j]\right)$. This matrix representation corresponds to sampling the two dimensional profile on the grid given by $I$. In general the choice of such grid is not unique, and the spacing of such grid may not be constant. This makes it difficult to define a general notion of distance between Euler Characteristic Surfaces matrices. For this reason we think it is more natural to define the Euler Characteristic Profile as a function like in \ref{def:ECP} and look for stability results in this setting.
\label{rmk:ECP_VS_ECS}
\end{remark}

%
%
%
%
%
%
%
%
%
%
%
%
%
%
%
%
\section{Stability of Euler Characteristic Curves and Profiles}

The goal of this section is to find a bound for the distance between Euler Characteristic Curves by some know topological quantity of the point cloud that is robust with respect to small perturbations of the point cloud. This way, the stability of Euler Characteristic Curves is obtained.

\subsection{Euler Characteristic Curves}
\label{sec:ECC_DIFF}

Since ECCs are are piece-wise constant functions, we consider the $L_1$ distances between them.

\begin{definition}
    Let $K_1$ and $K_2$ be two filtered cell complexes. 
    The $L_1$ distance between their Euler Characteristic Curves is 
    \[
      ||ECC(K_1, t) - ECC(K_2, t)||_1 = \int_\mathbb{R} |ECC(K_1, t) - ECC(K_2, t)| dt  \quad .
    \]
\label{def:ECC_W1}
\end{definition}

\begin{figure}
    \centering
    \scalebox{0.5}{\definecolor{cqcqcq}{rgb}{0.7529411764705882,0.7529411764705882,0.7529411764705882}
\definecolor{ffqqqq}{rgb}{1,0,0}
\definecolor{qqwuqq}{rgb}{0,0.39215686274509803,0}
\begin{tikzpicture}[line cap=round,line join=round,>=triangle 45,x=1cm,y=1cm]
\fill[line width=0pt,color=cqcqcq,fill=cqcqcq,fill opacity=0.3] (0,5) -- (0,4) -- (2,4) -- (2,5) -- cycle;
\fill[line width=0pt,color=cqcqcq,fill=cqcqcq,fill opacity=0.3] (2,4) -- (2,2) -- (4,2) -- (4,3) -- (3,3) -- (3,4) -- cycle;
\fill[line width=0pt,color=cqcqcq,fill=cqcqcq,fill opacity=0.3] (4,2.5) -- (4.993274931794596,2.5) -- (5,2) -- (4,2) -- cycle;
\fill[line width=0pt,color=cqcqcq,fill=cqcqcq,fill opacity=0.3] (5,3) -- (4.993274931794596,2.5) -- (6.000167128301285,2.5) -- (6,3) -- cycle;
\fill[line width=0pt,color=cqcqcq,fill=cqcqcq,fill opacity=0.3] (6.000167128301285,2.5) -- (6,1) -- (8.5,1) -- (8.5,2.5) -- cycle;
\draw [->,line width=2pt] (0,0) -- (10,0);
\draw [->,line width=2pt] (0,0) -- (0,6);
\draw [line width=2pt,color=qqwuqq] (0,5)-- (2,5);
\draw [line width=2pt,color=qqwuqq] (2,2)-- (5,2);
\draw [line width=2pt,color=qqwuqq] (5,3)-- (6,3);
\draw [line width=2pt,color=ffqqqq] (0,4)-- (3,4);
\draw [line width=2pt,color=ffqqqq] (3,3)-- (4,3);
\draw [line width=2pt,color=ffqqqq] (4,2.5)-- (8.5,2.5);
\draw [line width=2pt,color=ffqqqq] (8.5,1.02)-- (10,1.0193827843692975);
\draw [line width=2pt,color=qqwuqq] (6,1)-- (10,1);
\draw (4.521791218946116,-0.048293933392829155) node[anchor=north west] {Filtration};
\draw (-0.75,5.866910223704526) node[anchor=north west] {$\chi$};
\draw [fill=qqwuqq] (0,5) circle (2.5pt);
\draw [fill=qqwuqq] (2,2) circle (2.5pt);
\draw [fill=qqwuqq] (5,3) circle (2.5pt);
\draw [fill=qqwuqq] (6,1) circle (2.5pt);
\draw [fill=ffqqqq] (0,4) circle (2.5pt);
\draw [fill=ffqqqq] (3,3) circle (2.5pt);
\draw [fill=ffqqqq] (4,2.5) circle (2.5pt);
\draw [fill=ffqqqq] (8.5,1.02) circle (2.5pt);
\end{tikzpicture}}
    \caption{Two Euler Characteristic Curves in red and green. The absolute value of their difference is highlighted in shaded gray.}
    \label{fig:ECC_distance}
\end{figure}
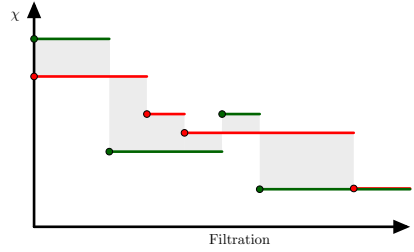

The proof presented in this section is inspired by the stability result for persistence functions by Chung and Lawson \cite{chung_persistence_2022}. They analyze stability of a wide class of persistence curves and obtain a general bound (see Theorem 1 in \cite{chung_persistence_2022}). However, trying to specialize this result to the simple Betti curve case leads a term that depends on the number of points in the Persistence Diagram. Hence the authors claim that Betti curves are unstable. 

We will instead carry out the proof focusing exclusively on Betti curves, by doing so a stability result can be obtained.

\begin{definition}
    Let $K$ be a cell complex with filtration function $f$. Its $k-$th Betti curve is a function that assigns to each filtration level the $k-$th Betti number of the corresponding subcomplex.
    \[
        \beta_k(K, t) = \beta_k(K_t) \quad .
    \]
\end{definition}

Let now $D$ be the $k-$dimensional persistence diagram obtained from a filtered complex $K$. The Fundamental Lemma of Persistent Homology \cite{edelsbrunner_computational_2022} states that the $k-$th Betti number of the subcomplex $K_t$ can be obtained by counting the points in the diagram that lie in the box $(x,y) \, | \, x \leq t < y$,
\[
\beta_k(K_t) = \#[(b,d) \in D \, | \, b \leq t < d] .
\]
We can reformulate this statement by assigning to each point $(b,d)$ in the diagram its indicator function in the interval $[b,d)$, $I_{[b,d)}(t)=1$ if $t \in [b,d)$ and $0$ otherwise. This indicator functions are exactly the bars in the barcode representation. By doing so we can define the $k$-dimensional Betti curve as the step function obtained by summing up all these indicator functions. 
\begin{definition}
The $k-$th Betti curve for a persistence diagram $D$ with finitely many off diagonal point is
\begin{equation*}
\beta_k(D, t) = \sum_{(b,d) \in D}  I_{[b,d)}(t) .
\end{equation*}
\end{definition}

\begin{proposition}
Let $C$ and $D$ be two $k$-dimensional persistence diagrams. Their Betti curves are stable with respect to the 1-Wasserstein distance,
\begin{equation}
    \norm{\beta_k(C, t) - \beta_k(D, t)}_1 \leq 2W_1(C,D).
\label{eq:betti_stability}
\end{equation}
\label{proof:betti}
\end{proposition}

\begin{proof}
Let us consider two $k-$dimensional persistence diagrams $C, D$ and assume the optimal matching under the 1-Wasserstein distance is known. Moreover let us index the points in each diagram as $(b_i^C, d_i^C)$ and $(b_i^D, d_i^D)$ so that points with matching indices are paired under the optimal matching. The case when points from one diagram are matched to diagonal is described in the case 2 below. We can then write the difference between the two Betti curves as following
\begin{align*}
    \norm{\beta_k(C, t) - \beta_k(D, t)}_1 &= \norm{\sum_i I_{[b_i^C,d_i^C)}(t) - I_{[b_i^D,d_i^D)} (t)}_1 \label{eq:start} \\
    & \leq \sum_i \norm{I_{[b_i^C,d_i^C)}(t) - I_{[b_i^D,d_i^D)} (t)}_1 .
\end{align*}

Let us focus on a single term of the sum, $\norm{h_i(t)} = \norm{I_{[b_i^C,d_i^C)}(t) - I_{[b_i^D,d_i^D)} (t)}_1 $. Then, one of the following cases have to hold:

\underline{Case 1:} $b_i^C \leq b_i^D \leq d_i^C \leq d_i^D$ .

\begin{figure}[!ht]
    \centering
    \definecolor{wrwrwr}{rgb}{0.3803921568627451,0.3803921568627451,0.3803921568627451}

\begin{tikzpicture}[line cap=round,line join=round,>=triangle 45,x=0.5cm,y=0.5cm]
\draw [line width=2pt] (0,0)-- (5,0);
\draw [line width=2pt] (3,1)-- (10,1);
\draw [->,line width=0.8pt,color=wrwrwr] (0,0) -- (12,0);
\begin{scriptsize}
\draw [fill=black] (0,0) circle (2.5pt);
\draw[color=black] (0.22,0.51) node {$b^C$};
\draw [color=black] (10,1) circle (2.5pt);
\draw[color=black] (10.22,1.51) node {$d^D$};
\draw [fill=black] (3,1) circle (2.5pt);
\draw[color=black] (3.22,1.49) node {$b^D$};
\draw [color=black] (5,0) circle (2.5pt);
\draw[color=black] (5.22,0.49) node {$d^C$};
\end{scriptsize}
\end{tikzpicture}
    \caption{Case 1}
    \label{fig:case2}
\end{figure}

\begin{align*}
     \norm{h_i(t)} &= \int_{b_i^C}^{b_i^D}{\abs{I_{[b_i^C,d_i^C)}(t)} dt} + \int_{b_i^D}^{d_i^C}{\abs{I_{[b_i^C,d_i^C)}(t) - I_{[b_i^D,d_i^D)}(t)} dt } + \int_{d_i^C}^{d_i^D}{\abs{I_{[b_i^D,d_i^D)}(t)} dt} \\
     &= \int_{b_i^C}^{b_i^D}{\abs{1} dt} + \int_{b_i^D}^{d_i^C}{\abs{1-1} dt } + \int_{d_i^C}^{d_i^D}{\abs{1} dt} \\
     &= \abs{b_i^D - b_i^C} + \abs{d_i^D - d_i^C} \\
     &\leq 2 \max(\abs{b_i^D - b_i^C}, \abs{d_i^D - d_i^C}) .
\end{align*} \\

\underline{Case 2:} $b_i^C \leq b_i^D \leq d_i^D \leq d_i^C$ .

\begin{figure}[!ht]
    \centering
    \definecolor{wrwrwr}{rgb}{0.3803921568627451,0.3803921568627451,0.3803921568627451}

\begin{tikzpicture}[line cap=round,line join=round,>=triangle 45,x=0.5cm,y=0.5cm]
\draw [line width=2pt] (0,0)-- (10,0);
\draw [line width=2pt] (3,1)-- (5,1);
\draw [->,line width=0.8pt,color=wrwrwr] (0,0) -- (12,0);
\begin{scriptsize}
\draw [fill=black] (0,0) circle (2.5pt);
\draw[color=black] (0.22,0.51) node {$b^C$};
\draw [color=black] (5,1) circle (2.5pt);
\draw[color=black] (5.22,1.51) node {$d^D$};
\draw [fill=black] (3,1) circle (2.5pt);
\draw[color=black] (3.22,1.49) node {$b^D$};
\draw [color=black] (10,0) circle (2.5pt);
\draw[color=black] (10.22,0.49) node {$d^C$};
\end{scriptsize}
\end{tikzpicture}
    \caption{Case 2}
    \label{fig:case3}
\end{figure}

\begin{align*}
     \norm{h_i(t)} &= \int_{b_i^C}^{b_i^D}{\abs{1} dt} + \int_{b_i^D}^{d_i^D}{\abs{1-1} dt } + \int_{d_i^D}^{d_i^C}{\abs{1} dt} \\
     &= \abs{b_i^D - b_i^C} + \abs{d_i^C - d_i^D} \\
     &\leq 2 \max(\abs{b_i^D - b_i^C}, \abs{d_i^D - d_i^C}) .
\end{align*}

The matching of one point $(b_i^C, d_i^C) \in C$ with a point in the diagonal of $D$ is a degenerate Case 2 with $b_i^C \leq b_i^D = d_i^D \leq d_i^C$. Note that, because of this, $C$ and $D$ are not required to have the same number of off-diagonal points. 
\\

\underline{Case 3:} $b_i^C \leq d_i^C \leq b_i^D \leq d_i^D$
\begin{figure}[ht!]
    \centering
    \definecolor{wrwrwr}{rgb}{0.3803921568627451,0.3803921568627451,0.3803921568627451}

\begin{tikzpicture}[line cap=round,line join=round,>=triangle 45,x=0.5cm,y=0.5cm]
\draw [line width=2pt] (0,0)-- (5,0);
\draw [line width=2pt] (7,1)-- (10,1);
\draw [->,line width=0.8pt,color=wrwrwr] (0,0) -- (12,0);
\begin{scriptsize}
\draw [fill=black] (0,0) circle (2.5pt);
\draw[color=black] (0.22,0.51) node {$b^C$};
\draw [color=black] (5,0) circle (2.5pt);
\draw[color=black] (5.22,0.51) node {$d^C$};
\draw [fill=black] (7,1) circle (2.5pt);
\draw[color=black] (7.22,1.49) node {$b^D$};
\draw [color=black] (10,1) circle (2.5pt);
\draw[color=black] (10.22,1.49) node {$d^D$};
\end{scriptsize}
\end{tikzpicture}
    \caption{Case 3}
    \label{fig:case1}
\end{figure}

This case will never happen as a better matching can always be obtained by matching both points to the diagonal, which is a degenerate Case 2.

We have that $\norm{h_i(t)} \leq 2\max(\abs{b_i^D - b_i^C}, \abs{d_i^D - d_i^C})$ holds for every $i$. We can then write the difference between two Betti curves as
\begin{align*}
    \norm{\beta_k(C, t) - \beta_k(D, t)}_1 &\leq \sum_i \norm{I_{[b_i^C,d_i^C)}(t) - I_{[b_i^D,d_i^D)} (t)}_1  \\
    &\leq \sum_i{2 \max(\abs{b_i^D - b_i^C}, \abs{d_i^D - d_i^C}) } \\
    &= 2W_1(C, D) .
\end{align*}
\end{proof}

Thanks to the Euler-Poincaré formula, the Euler Characteristic Curve of a filtered complex $K$ can be obtained as the alternating sum of its Betti curves.

\[
ECC(K,t) = \sum_k (-1)^k \beta_k(K,t).
\]

A stability result for the ECCs can be immediately derived from \ref{eq:betti_stability} assuming that the complex $K$ has nonzero persistence diagrams in a finite number of dimensions, each of them containing a finite amount of off-diagonal points.

\begin{proposition}
    Let $X$ and $Y$ be two filtered cell complexes.
    The $L_1$ difference between the Euler Characteristic Curves of $X$ and $Y$ is bounded by the sum of the 1-Wasserstein distances between the corresponding $k$-dimensional persistence diagrams $Dgm_k(X)$, $Dgm_k(Y)$.
    \begin{equation}
        \norm{ECC(X, t) - ECC(Y, t)}_1 \leq \sum_{k} 2 W_1(Dgm_k(X), Dgm_k(Y)) .
    \label{eq:ECC_stability}
    \end{equation}
    Where the sum is over all dimensions in which the persistence diagrams are non-empty.
\end{proposition}
\label{theo:ECC_stability}

\begin{proof}
It is an immediate consequence of \ref{proof:betti} and the triangular inequality.
\begin{align*}
        \norm{ECC(X, t) - ECC(Y, t)}_1 &= \norm{\sum_{k=0}^n (-1)^k ( \beta_k(Dgm_k(X), t) - \beta_k(Dgm_k(Y), t))}_1 \\
    &\leq \sum_{k=0}^n \norm{\beta_k(Dgm_k(X), t) - \beta_k(Dgm_k(Y), t)}_1 \\
    &\leq \sum_{k=0}^n 2 W_1(Dgm_k(X), Dgm_k(Y)) .
\end{align*} 
\end{proof}

The above Proposition \ref{theo:ECC_stability} is in explicit contrast with the claim that the Euler Characteristic Curve is unstable. In addition to the already mentioned work by Chung and Lawson \cite{Chevyrev_2020}, a similar statement can be found in \cite{beltramo_euler_2021} and \cite{Chevyrev_2020}.

\begin{remark}
    With reference to Figure \ref{fig:ECC_distance}, the left-hand side in \ref{theo:ECC_stability} is finite when the two ECCs agree from some filtration value onward. This is exactly what happens, for example, when considering curves obtained from full complexes, i.e. filtered complexes having a single simplex as a last element of a filtration: at some value all possible faces will have entered the filtration and so the Euler characteristic will stabilize at $1$. If this does not happen the difference between the two ECCs will be unbounded. At the same time, it is straightforward to show that if two filtered complexes have different Euler characteristic at $+\infty$ their homologies will have a different number of essential classes. This translates to a different number of points at infinity in the persistence diagrams, whose Wasserstein distance would then be unbounded. In this case, the above result will trivially be $+\infty \leq +\infty$.
\end{remark}

\subsection{Euler Characteristic Profiles}
We can immediately extend the notion of $L_1$ distances between ECCs to work in the general case of $n-$dimensional ECPs.

\begin{definition}
    Let $K_1, K_2$ be two multifiltered cell complexes. The $L_1$ distance between the corresponding $n-$dimensional Euler Characteristic Profiles is 
    \[
        ||ECP(K_1, v) - ECP(K_2, v)||_1 = \int_{\mathbb{R}^n} |ECP(K_1,v) - ECP(K_2,v)| dv  \quad .
    \]
\end{definition}

It is natural to ask whether the stability result in \ref{eq:ECC_stability} can be naturally extended to the multi-parameter case. In the existing literature, Chen et al. proposed the following weak $L_1$-metric in the case of bifiltered complexes (see Definition 3.2 in \cite{chen2022tampsgcnets}). Let us remind the proposed construction; consider two cell complexes $K_1$ and $K_2$ with a bifiltration function $F : K_{1,2} \rightarrow \mathbb{R}^2$. Let us denote with $f$ and $g$ the two real valued functions in the bifiltrations such that $F(\sigma) = ((f(\sigma), g(\sigma)))$ for every cell $\sigma$. Moreover, let us index the threshold values of $F$ as $I = \{(a_i, b_j) \, | \, 1 \leq i \leq m , 1 \leq j \leq n\}$. The idea behind Chen et al. construction is to fix one of of the two filtrations at a specific value and consider the distances between the single parameter persistence diagrams induced by the other filtration function. By considering the set of threshold values $I$ as a matrix with $i$ rows and $j$ columns , they define the \emph{$i^{th}$ column distance} for the $k-$dimensional PDs as $D^{i*}_k(K_1, K_2) = W_1 (D^g_k(K_1^{i*}), D^g_k(K_2^{i*}))$. Similarly, the \emph{$j^{th}$ row distance} is $D^{*j}_k(K_1, K_2) = W_1 (D^f_k(K_1^{*j}), D^f_k(K_2^{*j}))$. 

\begin{definition}[Definition 3.2 in \cite{chen2022tampsgcnets}]
    The weak $L_1$ metric between $K_1$ and $K_2$ is 
    \[
      D(K_1, K_2) = \max \{ \sum_{k=0}^M \sum_{i=1}^m D^{i*}_k (K_1, K_2), \: \sum_{k=0}^M \sum_{j=1}^n D^{*j}_k (K_1, K_2) \}  .
    \]
\label{def:L1_metric}
\end{definition}

Being able to recover the single parameter case, they prove the following stability result.

\begin{proposition}[Theorem 3.1 in \cite{chen2022tampsgcnets}]
    Let $K_1, K_2$ be two bifiltered cell complexes. The distance between the corresponding Euler Characteristic Surfaces is bounded by the weak $L_1$ metric metric between $K_1$ and $K_2$,
    \[
        ||ECP(K_1, v) - ECP(K_2, v)||_1 \leq c \cdot D(K_1, K_2) \: ,
    \]
    for some $c > 0$.
\label{prop:weak_stab}
\end{proposition}

This constructions appears to be the natural generalization to multifiltration case of the stability result in \ref{eq:ECC_stability}. However, there are some fundamental problems that undermine the usefulness of such weak $L_1$ metric.

\begin{remark}
    In our opinion the sums over rows or columns in \ref{def:L1_metric} should be replaced with integrals over the filtration ranges. As already discussed in Remark \ref{rmk:ECP_VS_ECS}, this would allow for more flexibility when dealing with filtration thresholds whose spacing is not constant.
\label{rmk:L1_metric}
\end{remark}

\begin{figure}[!ht]
    \centering
    \scalebox{0.5}{\definecolor{qqwuqq}{rgb}{0,0.39215686274509803,0}

\begin{tikzpicture}[line cap=round,line join=round,>=triangle 45,x=1cm,y=1cm]

\fill[line width=2pt,color=qqwuqq,fill=qqwuqq,fill opacity=0.1] (4,2) -- (8,2) -- (8,8) -- (4,8) -- cycle;
\fill[line width=2pt,color=red,fill=red,fill opacity=0.1] (2,2) -- (4,2) -- (4,8) -- (2,8) -- cycle;
\draw [->,line width=2pt] (0,0) -- (0,8);
\draw [->,line width=2pt] (0,0) -- (8,0);
\draw [line width=2pt,color=qqwuqq] (2,2)-- (8,2);
\draw [line width=2pt,color=qqwuqq] (2,8)-- (2,2);
\draw [line width=2pt, color=qqwuqq, dash pattern=on 1pt off 10pt] (4,8)-- (4,2);
\draw [line width=2pt] (2,2)-- (4,2);
\draw [fill=qqwuqq] (2,2) circle (2.5pt);
\draw[color=qqwuqq] (2.25, 2.25) node {g};
\draw [fill=qqwuqq] (4,2) circle (2.5pt);
\draw[color=qqwuqq] (4.25, 2.25) node {g'};
\draw[color=black] (3,1.75) node {\Large $\epsilon$};
\end{tikzpicture}}
    \caption{Minimal counterexample for the instability of ECP. Consider a cell complex made by only one vertex whose $+1$ contribution appears at some point $g=(g_1, g_2) \in \mathbb{R}^2$ and move it to $g'=(g_1+\epsilon, g_2)$. Their difference, the region shaded in red, is unbounded}
    \label{fig:ECP_instability}
\end{figure}
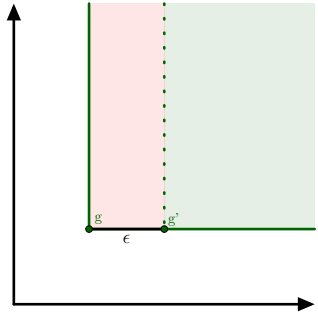

\begin{remark}
    Proposition \ref{prop:weak_stab} will evaluate to a trivial $\infty \leq \infty$ in most cases, even the simplest one. Consider for example the situation depicted in Figure \ref{fig:ECP_instability} of the ECP of a bifiltered complex $K_1$ made by just one 0-dimensional cell that appears at filtration value $(g_1, g_2)$. The ECP will then be $1$ in the cone $\{ (x,y) \in \mathbb{R}^2 : x \geq g_1 , \, y \geq g_2 \}$  and $0$ otherwise. We can obtain a different complex $K_2$ by perturbing the first filtration value by an $\epsilon$ amount $(g_1+\epsilon, g_2)$. The difference between the two ECPs will then be unbounded. At the same time, also the weak $L_1$ distance between $K_1$ and $K_2$ will be unbounded because in the interval $[g_1, g_1+\epsilon) \times [g_2, +\infty]$ the two complexes have a different number of essential classes and so the $W_1$ distance between the corresponding PDs will be infinite.
\label{rmk:ECP_instable}
\end{remark}

Because of the discussed issues, the stability result in~\cite{chen2022tampsgcnets}, while being formally correct, does not cover a lot of practically relevant cases.

However, in most applications we can \emph{truncate} the ECP, by limiting its filtration domain to i.e. the interval $[0, f_\infty]$ in every filtration dimension, where $f_\infty$ is a finite value. Note that this value at infinity should not be the same as the maximum filtration value of the complex's cells, but it should be strictly larger than the maximum filtration value. For example, in the case of images whose pixels have integer filtration values in the $[0, 255]$ range (see Section \ref{RGB}) we could choose $f_{\infty} = 256$ as truncation value. By doing so, the distance between every pair of ECP will be finite but it will of course depend on the truncation value. 
Using truncation, we can state the following result.

\begin{proposition}
    Let $K$ be a finite cell complex with a $n$-dimensional multifiltration $F: K \rightarrow \mathbb{R}^n$. We define $K^\epsilon$ as the complex obtained by perturbing the filtration values of each cell in $K$ by at most $\epsilon$ in in $l^\infty$ norm. Let us assume, for simplicity, that we truncate the domain of every filtration function to the same interval $[0, f_{\infty}]$.  We then have the following bound
    \[
    ||ECP(K, v) - ECP(K^\epsilon, v)||_1 \leq |K| \cdot d \cdot \epsilon^{n-1} \cdot f_{\infty} \quad ,
    \]

    where $|K|$ is the number of cells in the complex and $n$ is the number of filtration parameters.
\end{proposition}
\begin{proof}
    Let us consider a single cell $\sigma \in K$ with filtration value $g = (g_1, \cdots, g_d)$ . Its contribution to the ECP will be $(-1)^{dim(\sigma)}$ in the cone above $g$ (i.e. for all points $x \in \mathbb{R}^d$ such that $g \leq x$ coordinate-wise). Let $\sigma'$ be the corresponding cell in $K^\epsilon$ whose filtration values have been maximally perturbed to $g' = (g_1 + \epsilon, \cdots, g_d + \epsilon)$. The volume of the region which is in the cone of $g$ but not on the cone of $g'$ can be bounded by a sum of $n$ $n$-dimensional cuboids of base $\epsilon^{n-1}$ and height $f_{\infty}$, each of them corresponding to a shift of $\epsilon$ in the direction of one of the axis, $V_\sigma \leq d \cdot \epsilon^{n-1} \cdot f_{\infty}$, where the inequality is due to the fact that cuboids can have non-empty intersection.
    One of such cuboids is shaded in red in Figure \ref{fig:ECP_instability}. 
    Multiplying by the total number of cells give us the bound.
\end{proof}
%
%
%
%
%
%
%
%
%
%
%
%
%
%
%
%
%
%
%
\section{Algorithms}
\label{sec:Algo}

Recall that the Euler characteristic of a cell complex is the alternating sum of the number of its cells in each dimension. The contribution of each cell will thus be plus or minus one depending the dimension of the cell. Moreover, this contribution will appear at the cell's filtration level. Therefore, if we are able to obtain a list of all cells with their filtration values we can compute the Euler characteristic at each filtration level. This is the main idea behind the following algorithms, which will always return what we will denote as \textsc{list\_of\_contributions}, a list of pairs $(f(\sigma), (-1)^{dim(\sigma)})$ that stores each cell's contribution to the EC at the cell filtration level. Once this pairs have been sorted in ascending order with respect to the filtration, the Euler Characteristic Curve can be reconstructed by progressively summing up the contributions of following elements in the list.

\begin{remark}
    Roune and Sáenz de Cabezón \cite{roune_complexity_2011} proved that computing the Euler characteristic of a simplicial complex given by its vertices and facets is  is \#-P-complete. Even if their result does not mention filtered complexes, it follows from it that the problem of computing the ECC is at least P-complete. Otherwise, by contradiction, we could construct an arbitrary filtration of the considered complex, and look at the end value of the curve to obtain the Euler characteristic of the complex in polynomial time.
\end{remark}

\subsection{Vietoris-Rips complexes}
\label{sec:VR}
In this section we will present a distributed algorithm to compute the Euler Characteristic Curve of a Vietoris-Rips simplicial complex obtained from a collection of points in $\mathbb{R}^n$.

\begin{definition}
    Let $X$ be a finite collections of points in $\mathbb{R}^n$, also denoted as a point cloud.
    Given a parameter $\epsilon \leq 0$, the \emph{Vietoris-Rips complex} constructed from $X$ is the collection of all subset of diameter at most $2\epsilon$ , where the diameter is the greatest distance between any pair of vertices 
    \[
    \text{V-R}(X, \epsilon) = \{ \sigma \subseteq X \mid diam(\sigma) \leq 2\epsilon \} \quad .
    \]

    The filtration of each simplex is given by its diameter.
\end{definition} 

The Vietoris-Rips complex is a \emph{flag complex}, this means that a subset $S$ of vertices is in the complex if every pair of vertices in $S$ is in the complex.  This is analogous to saying that the Vietoris-Rips complex is completely determined by its 1-skeleton graph as there is a 1 to 1 correspondence between simplices in the complex and cliques in its 1-skeleton graph

Therefore, it is straightforward to see that listing all the simplices in a Vietoris-Rips complex is equivalent to perform a cliques count of its 1-skeleton graph \cite{boissonnat_simplex_2014}. In order to compute the contributions to the ECC we need to find an efficient and distributed way to list all cells in the simplex (i.e. all cliques in the 1-skeleton graph), and their filtration values (i.e. the length of the longest edge in each clique), this can be achieved in the following way. Given an ordered list of points $X = \{x_i : i \in [1, n] \}$\footnote{The points can be ordered in an arbitrary way.} and a maximum distance $\epsilon$, for each point $x_i$ we build its local graph $G_i$ of \emph{subsequent neighbours}, namely all points $x_j \in B(x_i, \epsilon) \cap X$ with $j>i$. For each $G_i$ we list all of its cliques that contain $x_i$. They will correspond to simplices with $x_i$ being the smallest vertex in the chosen ordering of points. This way each simplex $\sigma$ in the V-R complex will be generated exactly once, when considering the local graph of its lowest vertex in the considered ordering. 

Algorithm \ref{algo:local_contribution_VR}, that uses this idea, describes a way to list all the simplices of increasing dimension. At each iteration we obtain a list of $d$-dimensional simplices (given as collections of vertices) and, for each of them, a list of common subsequent neighbours of its vertices. We can then extend each simplex to a $(d+1)$-dimensional one by adding one common neighbour to the collections of vertices. When doing so we need to update the simplex's filtration value if one of the newly added edges is longer than the current filtration. Moreover, we need to update the list of common subsequent neighbours by intersecting it with the subsequent neighbours of the newly added vertex. Once we have obtained all possible $(d+1)$-simplices, we carry out this extension procedure one dimension higher. All of these operations are performed at the local graph of each vertex. The procedure ends when no simplex can be extended, i.e. when all maximal simplices have been listed.
This construction might be understood as a breadth-first traversal of the simplex tree \cite{boissonnat_simplex_2014}.  

The main advantages of the proposed algorithm are two: it does not require to construct the whole complex, leading to a significant decrease in memory utilization; it considers each point separately, allowing the computations to be carried out independently.

The inputs of our algorithm are $X$, a ordered list of points in $\mathbb{R}^n$ and a maximum filtration value $\epsilon$. The output is \textsc{list\_of\_contributions}, an ordered list of pairs. For each simplex $\sigma$ we store its contribution as a tuple ($(-1)^{dim(\sigma)}, f(\sigma)$). The output list will sorted according to the filtration values.


\begin{algorithm}[H]
\SetAlgoLined
  \KwInput{Ordered point cloud $X$, $\epsilon > 0$}
  \KwOutput{A ordered list of pairs (filtration, $\pm 1$)}
  Create an empty vector $C$ \\
 \For{every point $x_i$ in $X$}
    {
    create the local graph $G_i$ of subsequent neighbours of $x_i$\;

    simplices = [$x_i$] \;
    filtrations = [0] \;
    common$\_$subseq$\_$neighs = [[subseq$\_$neigh($G_i, x_i$)]] \;
    
    \While{simplices NOT empty} {
    \For{every simplex $\sigma$ $\in$ simplices} {
         add to $C$ the tuple  (filtration($\sigma$), $(-1)^{dim(\sigma)}$)\;
         }
    INCREASE$\_$DIMENSION($G_i$, simplices, common$\_$subseq$\_$neighs) \;
    }
    }
 sort $C$ according to the filtration value \;
 \Return{$C$} 
 \caption{COMPUTE LOCAL CONTRIBUTIONS V-R}
 \label{algo:local_contribution_VR}
\end{algorithm}

\begin{algorithm}[H]
\SetAlgoLined
  \KwInput{local graph $G_i$, simplices, common$\_$subseq$\_$neighs}
  new$\_$simplices = [] \;
  new$\_$filtrations = [] \;
  new$\_$common$\_$subseq$\_$neighs = [] \;
 \For{every simplex $\sigma$ $\in$ simplices}
    {
    \For{every $n\in$ common$\_$subseq$\_$neighs[$\sigma$]} {
         new$\_$simplices.append($\sigma$+[$n$])\;
         \BlankLine
         
         consider all the edges from vertices of $\sigma$ to $n$ and take the longest one \;
         new$\_$f = MAX( filtration($\sigma$) , length longest edge ) \;
         new$\_$filtrations.append( new$\_$f ) \;
         \BlankLine
         
         compute the intersection between the current common subsequent neighbours of $\sigma$ and the subsequent neighbours of $n$ in $G_i$ \;
         
         new$\_$common$\_$subseq$\_$neighs.append( intersection ) \;
         }
    }
  simplices = new$\_$simplices \;
  filtrations = new$\_$filtrations  \;
  common$\_$subseq$\_$neighs = new$\_$common$\_$subseq$\_$neighs \;
 \caption{INCREASE\_DIMENSION}
 \label{algo:increase_dim}
\end{algorithm}

Note that the Algorithm~\ref{algo:local_contribution_VR} is correct. Firstly, every simplex in the Vietoris-Rips complex will be generated. It will happen when its smallest vertex in the considered order will be considered in the \emph{for} loop. Secondly, each simplex will be generated only once in the \emph{INCREASE\_DIMENSION} procedure. A simplex $\sigma = [v_0,\ldots,n_{n-1},v_n]$, where $v_0 < \ldots < n_{n-1} < v_n$ will be generated from a simplex $[v_0,\ldots,n_{n-1}]$ by adding $v_n$ as a common neighbour of its vertices.

\subsection{Time performance}
The worst case scenario occurs when the the 1-skeleton graph is fully connected. Assuming the point cloud consist of $n$ points the resulting V-R complex will contain $2^n - 1$ simplices. In this case the time complexity of Algorithm \ref{algo:local_contribution_VR} is $\mathcal{O}(2^{n - 1} n)$. More details are provided in Appendix \ref{apdx:time}.

\subsection{Memory performance}
Assuming the worst case scenario, the size of the output list of contributions is $O(n^2)$ while the maximal memory required at one intermediate step is $O(2^n / \sqrt{n})$.
More details are provided in Appendix \ref{apdx:memory}.

\subsection{Choice of the vertex ordering}
Note that the total running time of the fully parallelized Algorithm \ref{algo:local_contribution_VR} can be dominated by few vertices whose simplex tree is considerably larger than the others. This explains the plateau in Figure \ref{fig:num_cores}. This effect can be mitigated by choosing a different ordering of the vertices. One efficient choice is to order the vertices by increasing number of $\epsilon$-neighbours. Since the local graph for each vertex is constructed by considering only its subsequent neighbours, this ordering will produce more evenly-sized simplex trees. A simple example is showed in Figure \ref{fig:reordering_example} while the effect of this reshuffling on a larger dataset is shown in Figure \ref{fig:reordering_exp}.
\begin{figure}
    \centering
    \includegraphics[width=0.4\textwidth]{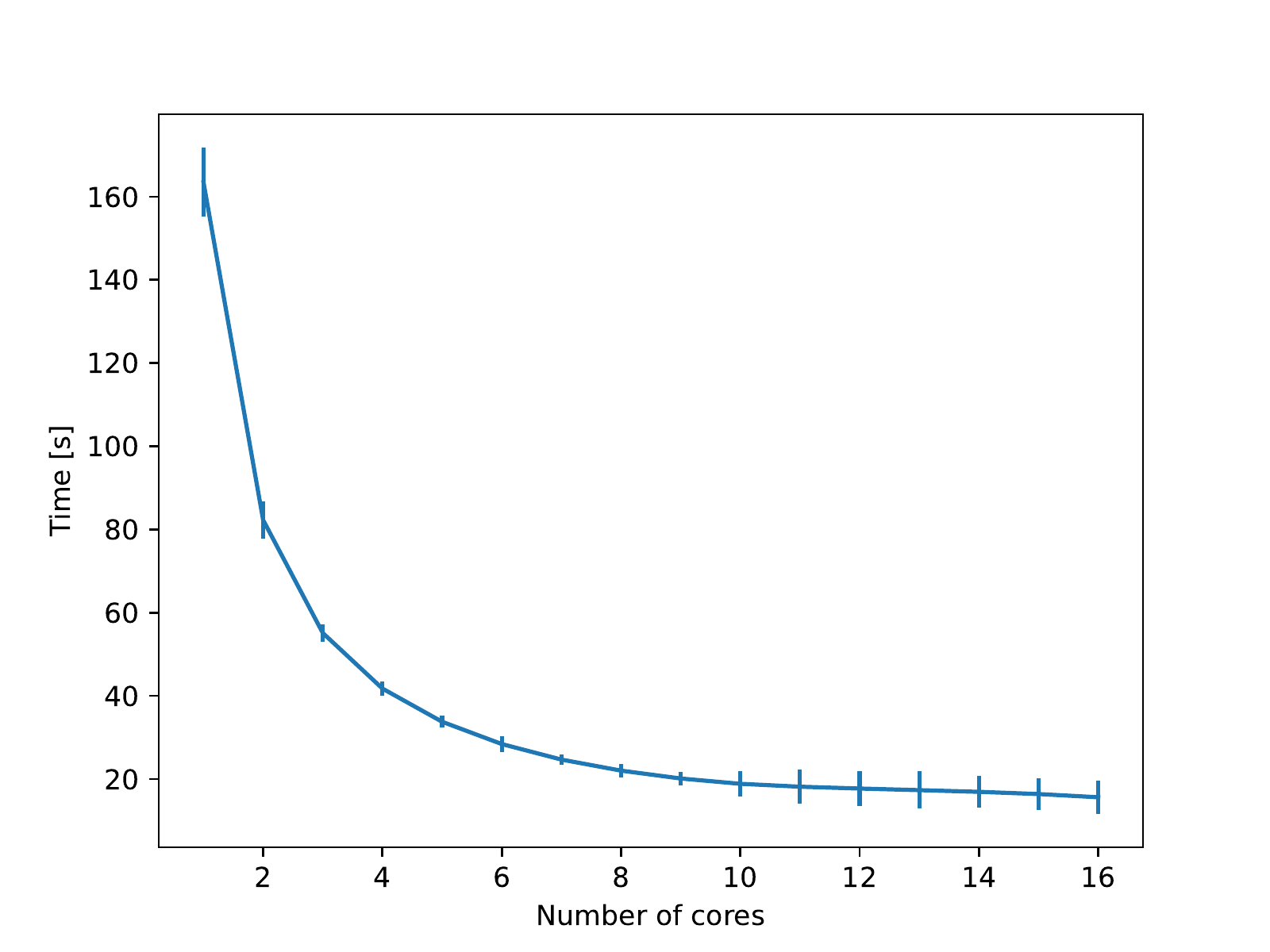}
    \caption{Average runtime over 10 runs of Algorithm \ref{algo:local_contribution_VR} as a function of the number of cores used. Contributions computed for the V-R complex obtained from 10000 points sampled from the unit 4-sphere up to a maximum radius of 0.4. Experiment run on a AMD Ryzen Threadripper PRO 5955WX cpu. Error bars are scaled up by a factor of 20 for visibility.}
    \label{fig:num_cores}
\end{figure}

\begin{figure}[!ht]
    \centering
    \hspace{2cm}
        \begin{subfigure}[t]{0.3\textwidth}
            \scalebox{0.7}{\begin{tikzpicture}[line cap=round,line join=round,>=triangle 45,x=1cm,y=1cm]

\draw [line width=1pt] (-1,-1)-- (0,0);
\draw [line width=1pt] (0,0)-- (2,1);
\draw [line width=1pt] (-1.42,1)-- (0,0);
\draw [line width=1pt] (-1.42,1)-- (-1,-1);

\draw [fill=black] (0,0) circle (2.5pt);
\draw[color=black] (0.1,0.3) node {$A$};
\draw [fill=black] (-1.42,1) circle (2.5pt);
\draw[color=black] (-1.3316407791065847,1.2438670270583445) node {$B$};
\draw [fill=black] (-1,-1) circle (2.5pt);
\draw[color=black] (-1,-1.3) node {$C$};
\draw [fill=black] (2,1) circle (2.5pt);
\draw[color=black] (2.0877585069494087,1.2438670270583445) node {$D$};

\fill[line width=0.5pt,fill=black,fill opacity=0.3] (0,0) -- (-1.42,1) -- (-1,-1) -- cycle;

\end{tikzpicture}} 
        \end{subfigure}
        \hspace{-1cm}
        \begin{subfigure}[t]{0.6\textwidth}
           \begin{forest}
[
  [A, no edge
    [AB
      [ABC]
    ]
    [AC]
    [AD]
  ]
  [B, no edge
    [BC
    ]
  ]
  [C, no edge
  ]
  [D, no edge]
]
\end{forest} 
        \end{subfigure}

    \hspace{2cm}
        \begin{subfigure}[t]{0.3\textwidth}
            \scalebox{0.7}{\begin{tikzpicture}[line cap=round,line join=round,>=triangle 45,x=1cm,y=1cm]

\draw [line width=1pt] (-1,-1)-- (0,0);
\draw [line width=1pt] (0,0)-- (2,1);
\draw [line width=1pt] (-1.42,1)-- (0,0);
\draw [line width=1pt] (-1.42,1)-- (-1,-1);

\draw [fill=black] (0,0) circle (2.5pt);
\draw[color=black] (0.1,0.25) node {$D$};
\draw [fill=black] (-1.42,1) circle (2.5pt);
\draw[color=black] (-1.3316407791065863,1.2438670270583445) node {$B$};
\draw [fill=black] (-1,-1) circle (2.5pt);
\draw[color=black] (-1,-1.3) node {$C$};
\draw [fill=black] (2,1) circle (2.5pt);
\draw[color=black] (2.087758506949406,1.2438670270583445) node {$A$};

\fill[line width=0.5pt,fill=black,fill opacity=0.3] (0,0) -- (-1.42,1) -- (-1,-1) -- cycle;

\end{tikzpicture}} 
        \end{subfigure}
        \hspace{-1cm}
        \begin{subfigure}[t]{0.6\textwidth}
           \begin{forest}
[
  [A, no edge
    [AD]
  ]
  [B, no edge
    [BC
      [BCD]
    ]
    [BD]
  ]
  [C, no edge
    [CD]
  ]
  [D, no edge]
]
\end{forest} 
        \end{subfigure}
    \caption{Different ordering of the vertices can produce different simplex trees. In the first row vertices are ordered by decreasing number of neighbours, in the second row by increasing number. The second choice produces more evenly-sized trees.}
    \label{fig:reordering_example}
\end{figure}
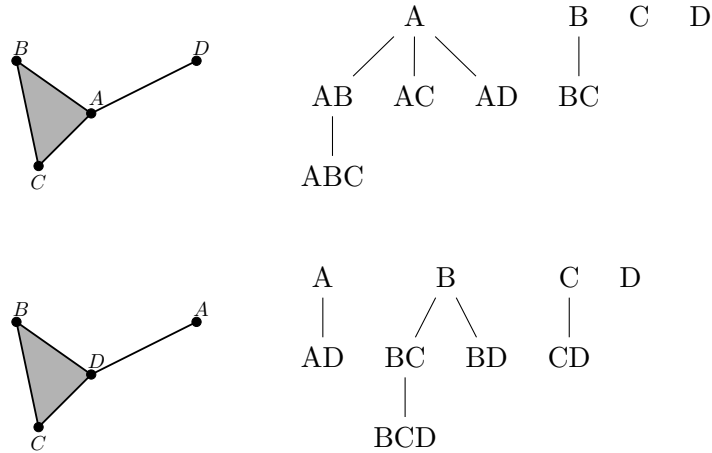

\begin{figure}
\begin{subfigure}[t]{0.45\textwidth}
    \centering
    \includegraphics[width=0.7\textwidth]{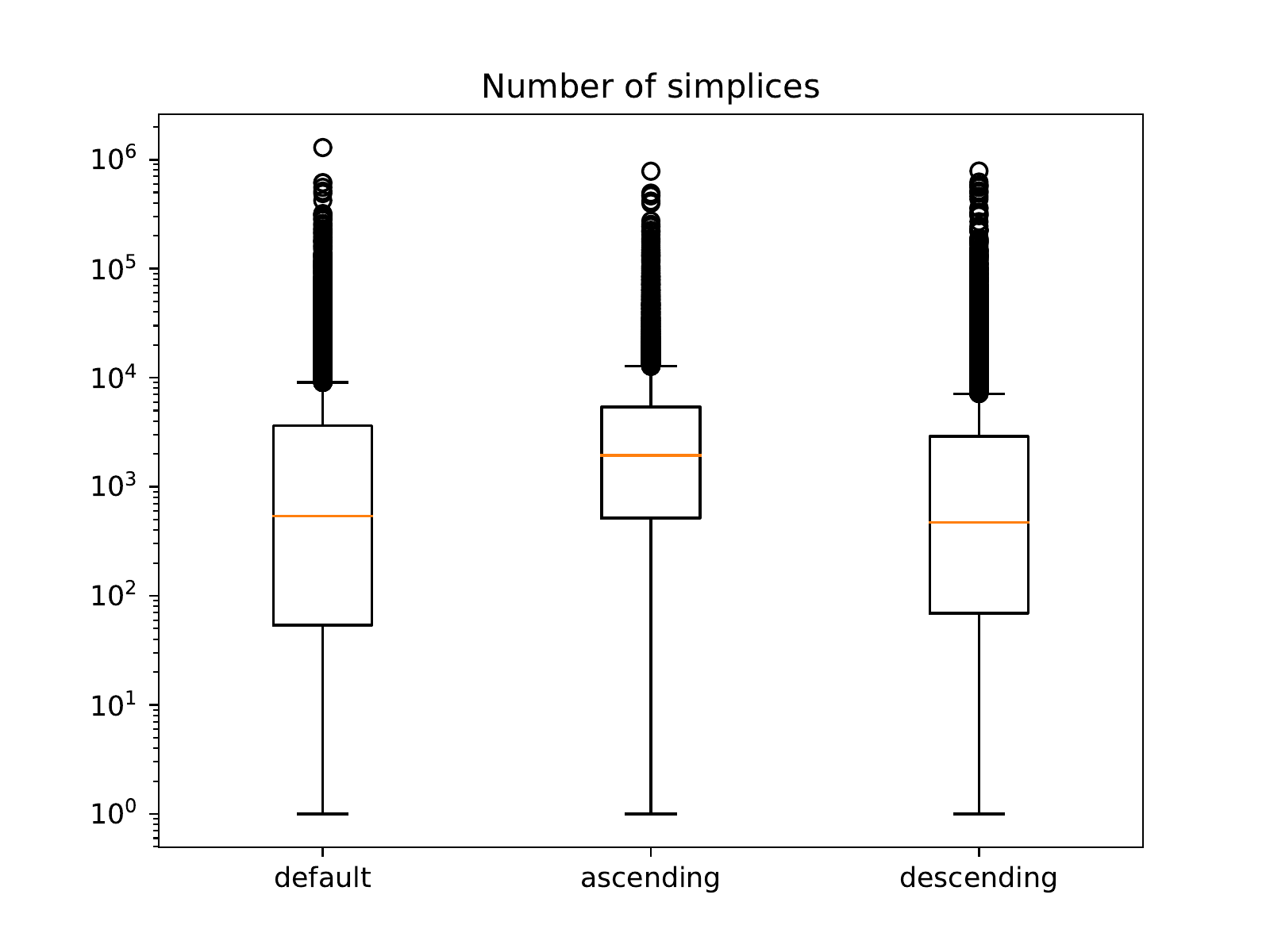}
\subcaption{}
\end{subfigure}
\hspace{-0.5cm}
\begin{subfigure}[t]{0.45\textwidth}
    \centering
    \includegraphics[width=0.7\textwidth]{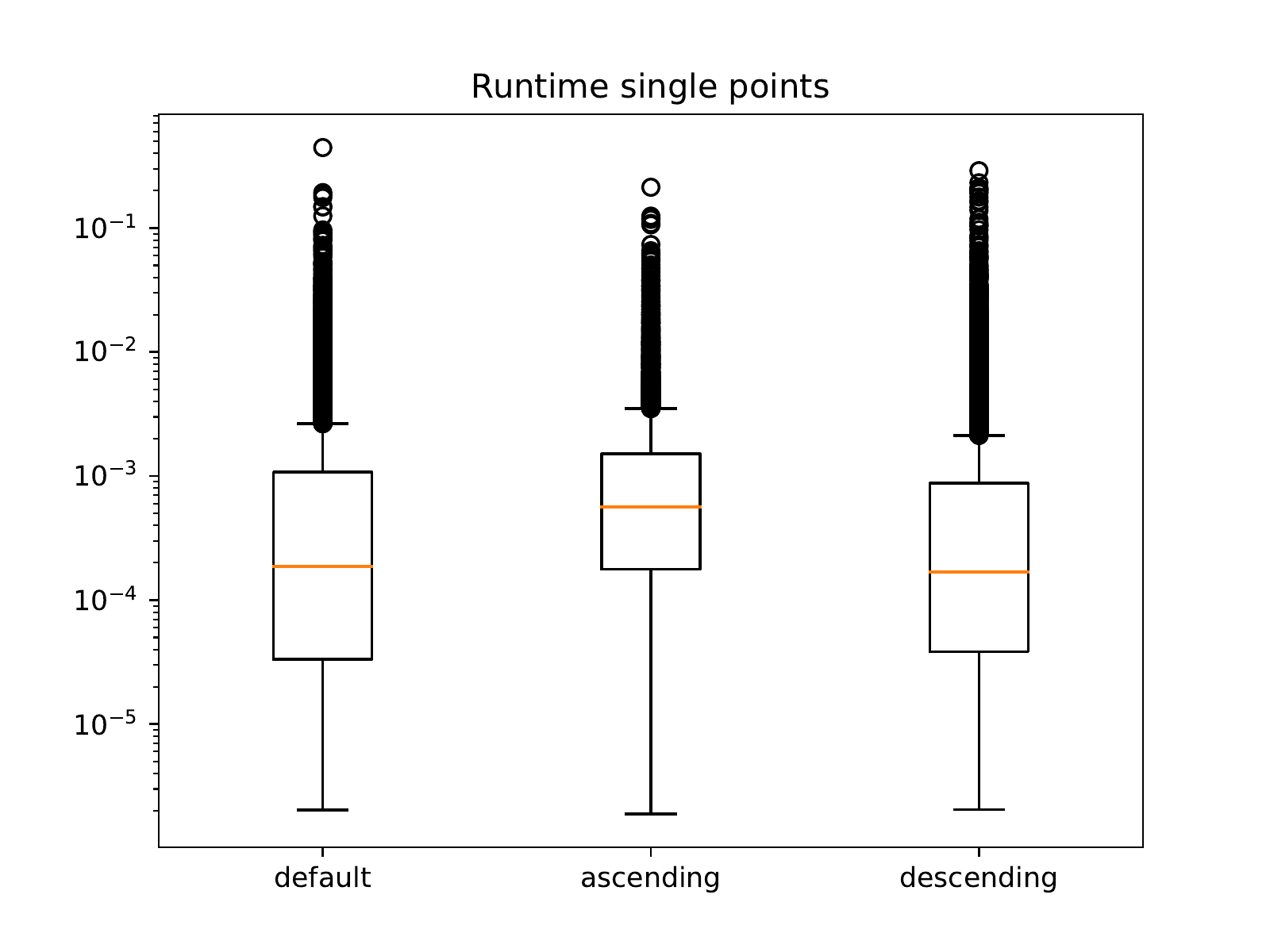}
\subcaption{}
\end{subfigure}
\caption{Effect of different orderings of vertices for the example in Figure \ref{fig:num_cores}. Selecting the ascending order allows to achieve a more even distribution in the number of simplices in each tree (panel A) thus reducing the number of very large trees which dominate the running time (panel B).}
\label{fig:reordering_exp}
\end{figure}


\newpage
\subsection{Cubical complexes}
\label{sec:cube}

Cubical complexes are the most used combinatorial structure to represent digital grayscale images and extract topological information from them. There are two ways to construct a cubical complex from an image, the V-construction and the T-construction. The former identifies pixels - also know as voxels, in case of images of arbitrary dimension - with the vertices (the 0-dimensional cells) of the cubical complex. Voxels's values are used to define the filtration on the vertices and the filtration of each other elementary cube is the maximal value of its vertices. The T-construction can be seen as the dual procedure, voxels's values are assigned to the top dimensional cubes and the filtration values are propagated to lower dimensional cells by taking the minimum over the cofaces. The relation between these two constructions is explored in a recent work by Bleile et al. \cite{bleile_persistent_2022}. In this paper we choose the T-construction, although the presented techniques translate easily to the V-construction.

Similar to the V-R case,  we are interested, given a grayscale $n$-dimensional image, in obtaining a list of contributions to the Euler characteristic of its corresponding cubical complex. As before, we then need to iterate over all cells $\sigma$ in the complex and store each contribution as a tuple $(f(\sigma), (-1)^{dim(\sigma)})$. This can be achieved in a streaming fashion by loading into memory a two voxel high slice of the image, iterating through the cells in the bottom row computing their contributions, and then moving the sliding window up by one voxel. To make sure we consider each cells contribution exactly once, at each iteration we consider one voxel and compute the contributions to the Euler characteristic of the cells in its \emph{upper closure}. Assuming that we can identify each top dimensional cell $c_i$ with the indices $(x_1, \cdots, x_n)$ of the corresponding voxel in the input $n$-dimensional image, we define the upper closure of $c_i$ as the set containing $c_i$ and all its faces that are shared with other top dimensional cells $c_j$ whose indices are $y_i=x_i$ or $y_i = x_i + 1$ for all $i$. An example of this procedure can be found in Figure \ref{fig:slice}.

\begin{algorithm}[H]
\SetAlgoLined
  \KwInput{A two voxels tick slice of an image, padded with $+\infty$}
  \KwOutput{A ordered list of pairs (filtration, $\pm 1$)}
  Create an empty vector $C$ \\
 \For{every voxel $c_i$ in the bottom row}
    {
    \For{every cell $\sigma$ in the upper closure of $c_i$} {
         add to $C$ the tuple  (filtration($\sigma$), $(-1)^{dim(\sigma)}$)\\
         }
    }
 sort $C$ according to the filtration value \\
 \Return{$C$} 
 \caption{COMPUTE LOCAL CONTRIBUTIONS CUBICAL}
 \label{algo:local_contribution_cube}
\end{algorithm}

\begin{figure}[!ht]
    \centering
    \includegraphics[width=0.3\textwidth]{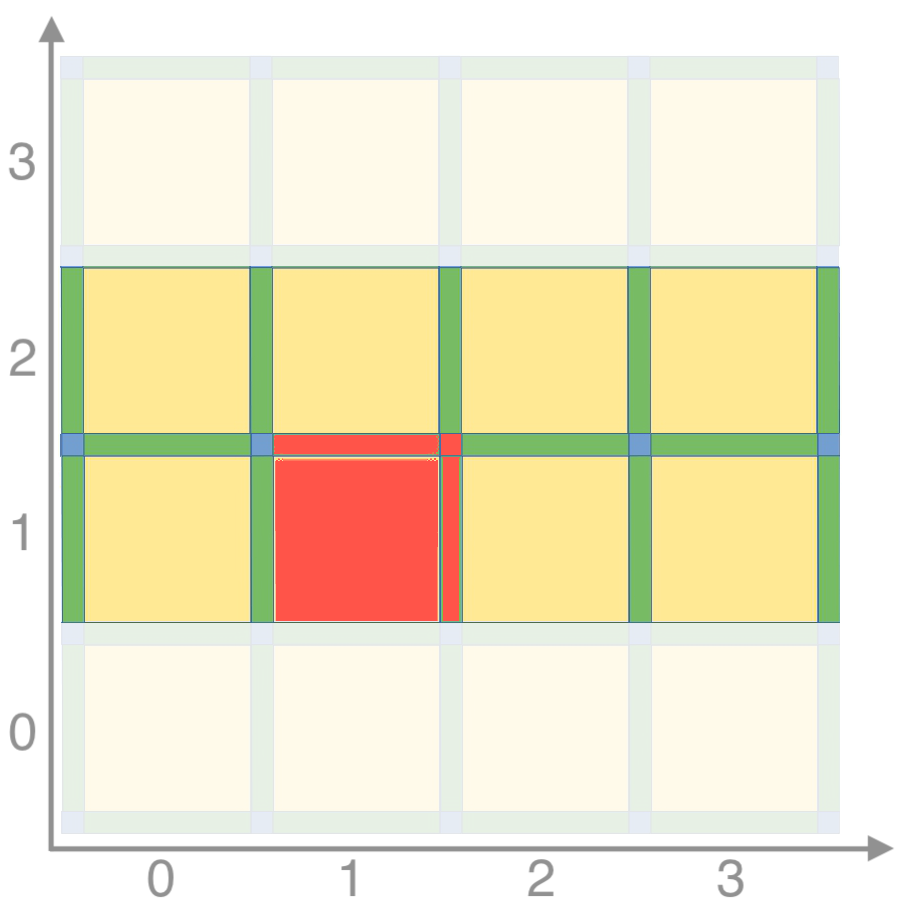}
    \caption{A slice of a cubical complex obtained from a 2 dimensional image. The image's pixels are associated to the top dimensional cells, depicted in yellow. Algorithm \ref{algo:local_contribution_cube} takes as input a two voxel tick slice of the image and iterates through the voxels in the bottom row. At each iteration a voxels is selected and the contributions of the cells in its upper closure are computed. In this example, the voxel at coordinates (1, 1) is selected and the considered contributions are depicted in red: the one coming from the corresponding 2-cell, the two from the 1-cells shared with (2,1) and (1,2) and the contribution from the 0-cell shared with (2,1) , (1,2) and (2,2).}
    \label{fig:slice}
\end{figure}

As already mentioned in Section \ref{sec:related_work}, a similar streaming algorithm to compute the ECC of grayscale images has been presented by Heiss and Wagner \cite{heiss_streaming_2017}. They also provide a fast open-source C++ implementation at \url{https://bitbucket.org/hubwag/chunkyeuler}. Recently Wang, Wagner and Chen \cite{wang_gpu_2022} provided a GPU implementation of the same algorithm available at \url{https://github.com/TopoXLab/GPU_ECC_SoCG2022}. However, there is a significant difference between their approach and the one we describe in Algorithm \ref{algo:local_contribution_cube}: they keep track of the faces introduced by each voxel by looking at the gray values of the voxel’s $3^d - 1$ neighbor and store the \emph{cumulative} change in the EC at the voxel's filtration value. This approach can not be generalized to the multiparameter filtration case as a cell could inherit different filtration values from different voxels. There are some small differences in the implementation too: CHUNKYEuler only works with integer filtration values and only accepts 'raw' binary files as input. Our implementation, while being not as fast as CHUNKYEuler, offers the user more flexibility in the input and choice of filtration (or multifiltration) values.

\subsection{Time and memory complexity}
Considering a $d$-dimensional image with $n$ voxels as input, the resulting cubical complex will have $3^d n$ cells. The running time of Algorithm \ref{algo:local_contribution_cube} is then linear in the number of cells in the complex with a multiplicative constant which is exponential in the dimension. This is not a problem in practice as images with dimension larger than 3 are not common in applications.
The memory requirement is just the space needed to store a two rows slice of the input image, the memory overhead for computing the local contributions for each voxel is negligible.

\subsection{From Euler Characteristic Curves to Profiles}
Both Algorithm \ref{algo:local_contribution_VR} and Algorithm \ref{algo:local_contribution_cube} can be immediately extended to compute the Euler Characteristic Profile of multifiltered Vietoris-Rips or cubical complexes. In the Vietoris-Rips case we require that all filtration functions should be defined on the vertices or the edges and then be extended to higher dimensional simplices by some user defined rule. This is to assure that the resulting multifiltered V-R complex is still a flag complex. In the case of cubical complexes we assume that the input images contains a $n-$tuple of numbers in each voxel - RGB images are a typical $n=3$ example - and values are propagated to lower dimensional cells by some user defined rules. In both cases the output of both algorithms will be a list of $(n+1)-$tuples $(f_1(\sigma), \cdots, f_n(\sigma), (-1)^{dim(\sigma)})$ that stores the list of contributions to the ECP at different points $f(\sigma) \in \mathbb{R}^n$. 

\begin{remark}
In above, the simplest case of so called $1$-critical multifiltration is discussed. In this case, each cell $\sigma$ appear in a unique value of the multifiltration. In a general case, a cell $\sigma$ may appear in multiple non-comparable values $p_1,\ldots,p_k$ of multifiltration. A simple generalization described below allows to adopt this presented algorithm to the general case; Let us assume that each $p_i$ is $n$ dimensional tuple, $p_i = (p_i^0,p_i^1,\ldots,p_i^n)$. We assume that $p_i$ and $p_j$ are not comparable provided $i \neq j$. It means that there exist a pair of coordinates $l \neq m$ so that $p_i^l < p_j^l$ and $p_i^m > p_j^m$. Then, the cell $\sigma$ contributes the value $(-1)^{\dim(\sigma)}$ for all the points $x \in \mathbb{R}^n$ for which there exist $i$ such that $x > p_i$. Note that the regions consisting of points greater that $p_i$ overlap for different $i \in \{1,\ldots,k\}$, hence we need to avoid double and multiple counting of the contributions. Below we describe a procedure to achieve it and enforce the contribution of exactly $(-1)^{\dim(\sigma)}$ for all $x > p_i$ for arbitrary $i \in \{1,\ldots,k\}$. For that purpose, given $i \neq j$, we define $p_i \vee p_j = ( \max(p_i^1,p_j^1),\max(p_i^2,p_j^2),\ldots,\max(p_i^n,p_j^n) )$.  Algorithm~\ref{algo:contributions_to_ECP}  define a set of points with appropriate contributions to enforce the required condition for all $x \geq p_i$ for all $i \in \{1,\ldots,k\}$.

\begin{algorithm}[H]
\SetAlgoLined
  \KwInput{$d$ - dimension of $s$, $p_1,\ldots,p_k \in \mathbb{R}^n$ - incompatible times of appearance of $\sigma$ in the multifiltration}
  \KwOutput{A collection of contributions of $\sigma$ to ECP}
  List $Contribution \leftarrow (p_i,(-1)^d)$, for $i \in \{1,\ldots,k\}$\\
  $P = p_i \vee p_j$ for every $i,j \in \{1,\ldots,k\}$\\
  Queue $L \leftarrow p_i \vee p_j$ for $i \neq j \in \{1,\ldots,k\}$\\
 \While{$L \neq \emptyset$}
    {
    p = dequeue(L)\\
    $P'$ = all elements $p' \in P$ such that $p' < p$\\
    \eIf{all elements $P'$ are already in $Contribution$}
    {
        $c =$ sum of values of elements in $P'$ in $Contribution$\\
        $Contribution \leftarrow (p,(-1)^d-c)$\\
    }
    {
        L = enqueue(p)}
    }

 \Return{$Contribution$} 
 \caption{CONTIBUTION OF $\sigma$ TO ECP}
 \label{algo:contributions_to_ECP}
\end{algorithm}

It is straightforward to see that for any given cell $\sigma$, its contributions to the ECP will change at at most in $p_i \vee p_j$ for $i,j \in \{1,\ldots,k\}$, where $\{p_1,\ldots,p_k\}$ are incompatible points in which $\sigma$ appears in the multifiltration. Algorithm~\ref{algo:contributions_to_ECP} scans all those points, and assigns the appropriate value (see line 1 and 9) to contributions to the ECP. Note that all points $p_1,\ldots,p_k$ have their contributions initially set in the line~1. 
Consequently, the presented algorithm will terminate, as in each iteration at least one $p$ will be added to the $Contribution$ list.  In addition, it explicitly enforces the correct contribution of the cell $\sigma$ to all points $x \geq p_i$ for any $i \in \{1,\ldots,k\}$. 

\end{remark}


\section{Data Structures for ECPs}

All the algorithms we described in the previous section output a list of contributions to the Euler Characteristic Profile. For a $n$-dimensional profile, each contribution in the list is a pair where the first entry is a $n$-tuple storing the coordinates in $\mathbb{R}^n$ at which the Euler characteristic varies by the integer values stored in the second item. When dealing with one dimensional ECCs it makes sense to sort the contributions according to their filtration value, in order to perform faster operations on them.

\subsection{Retrieving the EC at some filtration values}
\label{sec:retrieving_EC}
Given a ECP as a list of contributions, the first basic operation is to retrieve the value of the Euler characteristic at an arbitrary filtration value $f_*$. It can be obtained by summing up all the contributions in the ECP that appear at filtration values less or equal $f_*$. For a $d$-dimensional ECP this can be achieved in linear time with respect to the size of the contribution list. 
%
%
In the one dimensional case, we can take advantage of the total ordering on the list of contributions, since the filtration values $f_i \in \mathbb{R}$. By doing so we can build an auxiliary data structure storing the value of the Euler characteristic at each $f_i$, the points in which the ECC is changing value. This can be done in $O(n)$ time and space, where $n$ is the length of the list of contributions. Given such a structure, computing the value of the ECC at a given filtration $f_*$ boils down to the the search for the largest jump point $f_i < f_*$ and retrieving the value of the ECC therein. This can be achieved by interpolation search in $O(log(log(n)))$ time.

\subsection{Computing distances}

\subsubsection{Distances between Euler Characteristic Curves}
In Section \ref{sec:ECC_DIFF} we introduced the notion of difference between two ECCs, expressed in terms of the $L_1$ norm of the difference between the two curves. One should note that, in the case of finite Vietoris-Rips or cubical complexes, such a difference is always finite (but not bounded) as all ECCs will eventually stabilize to $1$ for a sufficiently large filtration value. In case when the construction of a Vietoris-Rips complexes is stopped at a certain diameter $2\epsilon$, and the final complexes have more than one infinite homology, it make sense to restrict the integral used in distance computations to an interval $[0,2\epsilon]$ in order to make the distances between the ECCs finite. 

Both Algorithm \ref{algo:local_contribution_VR} and \ref{algo:local_contribution_cube} return the computed ECC as list of pairs $(f_i, c_i)$ where $c_i$ is an integer representing the change in the Euler characteristic at filtration $f_i$. Such list is sorted in increasing order with respect to the filtration values. Using such data structure the difference between two ECCs can be computed in linear time with the size of the lists. Given two list of contributions $ECC_1$ and $ECC_2$ we can merge them in linear time, preserving the order. While merging we flip the sign of all the contributions coming from $ECC_2$. Let us denote the obtained list with $ECC_{1-2}$. Now the difference can be computed by iterating over the full list
\[
||ECC_1 - ECC_2||_1 = \sum_i (f_{i+1} - f_i) EC(f_i) \quad ,
\label{eq:diff_ECC}
\]

where $EC(f_i) = \sum_{j=0}^i c_j$ with respect to the ordering of $ECC_{1-2}$.

\subsubsection{Distances between Euler Characteristic Profiles}
Unfortunately the strategy proposed in the previous section is difficult to generalize in the multifiltration setting as there is no natural way to sort the list of contributions. We present here a basic algorithm to compute the distances between two ECPs and leave the search for potentially faster algorithm to future work.

Let $ECP_1$ and $ECP_2$ be two list of contributions representing two $n$-dimensional profiles. We can merge them in linear time, as in the one dimensional case, flipping the sign of the contributions in the second list. Let $N$ be the total number of contributions. With reference to Figure \ref{fig:ecp_grid}, the coordinates of such contributions will create a $n$-dimensional irregular grid of size $(N+1)$. The value of the EC inside each cuboid will be equal to the EC at the cuboid's bottom left corner and can be computed in $O(N)$. The $L_1$ distance between the two ECPs can then be obtained by summing up the values of the EC in each cuboid weighted by the cuboid's volume. Given that the number of cuboids is $(N+1)^d$, this operation can be computed in $O(N^{d+1})$. Note that the ECPs need to be truncated in order to avoid cuboids with infinite volume.

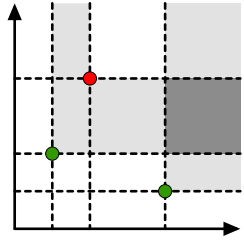
\begin{figure}
    \centering
    \scalebox{0.5}{\definecolor{uququq}{rgb}{0.25098039215686274,0.25098039215686274,0.25098039215686274}
\definecolor{aqaqaq}{rgb}{0.6274509803921569,0.6274509803921569,0.6274509803921569}
\definecolor{ffqqqq}{rgb}{1,0,0}
\definecolor{ttzzqq}{rgb}{0.2,0.6,0}

\begin{tikzpicture}[line cap=round,line join=round,>=triangle 45,x=1cm,y=1cm]
\fill[line width=2pt,color=aqaqaq,fill=aqaqaq,fill opacity=0.3] (1,6) -- (1,2) -- (4,2) -- (4,4) -- (2,4) -- (2,6) -- cycle;
\fill[line width=2pt,color=aqaqaq,fill=aqaqaq,fill opacity=0.3] (4,2) -- (4,1) -- (6,1) -- (6,2) -- cycle;
\fill[line width=2pt,color=aqaqaq,fill=aqaqaq,fill opacity=0.3] (4,4) -- (4,6) -- (6,6) -- (6,4) -- cycle;
\fill[line width=2pt,color=uququq,fill=uququq,fill opacity=0.6] (4,4) -- (4,2) -- (6,2) -- (6,4) -- cycle;
\draw [->,line width=2pt] (0,0) -- (0,6);
\draw [->,line width=2pt] (0,0) -- (6,0);
\draw [line width=2pt,dash pattern=on 4pt off 4pt] (1,6)-- (1,0);
\draw [line width=2pt,dash pattern=on 4pt off 4pt] (4,6)-- (4,0);
\draw [line width=2pt,dash pattern=on 4pt off 4pt] (0,2)-- (6,2);
\draw [line width=2pt,dash pattern=on 4pt off 4pt] (0,1)-- (6,1);
\draw [line width=2pt,dash pattern=on 4pt off 4pt] (0,4)-- (6,4);
\draw [line width=2pt,dash pattern=on 4pt off 4pt] (2,6)-- (2,0);

\draw [fill=ttzzqq] (1,2) circle (5pt);
\draw [fill=ttzzqq] (4,1) circle (5pt);
\draw [fill=ffqqqq] (2,4) circle (5pt);

\end{tikzpicture}}
    \caption{Example of a two dimensional ECP with three contributions. The green points indicate a $+1$ while the red point is a $-1$. The plane can then be subdivided in a $4 \times 4$ irregular grid. The coloring of each block indicates the value of the EC in that that region, white is $0$, light gray is $1$ and dark gray is $2$.}
    \label{fig:ecp_grid}
\end{figure}

\section{Vectorization}
Vectorizing the ECC / ECP is a critical step if we are interested in using these invariants in a Machine Learning framework.

\subsection{Curves}
Assume we are given an ECC whose filtration values ranges from $0$ to $f_{max}$. We can convert it to a vector by evenly sampling it $N$ times between $0$ and $f_{max}$. If we chose to include the endpoints the resulting vector will be $vec(ECC, N) = [EC(0), EC(\Delta), EC(2\Delta), \cdots, EC((N-2)\Delta), EC(f_{max})]$ , where $\Delta$ is the vectorization's resolution which is defined as $\Delta = f_{max} / (N-1)$.

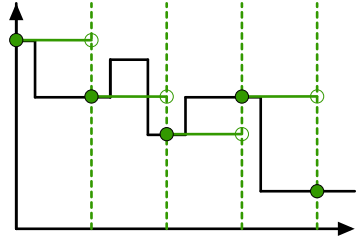
\begin{figure}[!ht]
    \centering
    \scalebox{0.5}{\definecolor{green}{rgb}{0.2,0.6,0}
\begin{tikzpicture}[line cap=round,line join=round,>=triangle 45,x=1cm,y=1cm]

\draw [->,line width=2pt] (0,0) -- (0,6);
\draw [->,line width=2pt] (0,0) -- (9,0);
\draw [line width=2pt,dash pattern=on 4pt off 4pt,color=green] (2,0)-- (2,6);
\draw [line width=2pt,dash pattern=on 4pt off 4pt,color=green] (4,0)-- (4,6);
\draw [line width=2pt,dash pattern=on 4pt off 4pt,color=green] (6,0)-- (6,6);
\draw [line width=2pt,dash pattern=on 4pt off 4pt,color=green] (8,0)-- (8,6);
\draw [line width=2pt] (0,0)-- (0,6);
\draw [line width=2pt] (0,5)-- (0.5,5);
\draw [line width=2pt] (0.5,5)-- (0.5,3.5);
\draw [line width=2pt] (0.5,3.5)-- (2.5,3.5);
\draw [line width=2pt] (2.5,3.5)-- (2.5,4.5);
\draw [line width=2pt] (2.5,4.5)-- (3.5,4.5);
\draw [line width=2pt] (3.5,2.5)-- (3.5,4.5);
\draw [line width=2pt] (3.5,2.5)-- (4.5,2.5);
\draw [line width=2pt] (4.5,3.5)-- (4.5,2.5);
\draw [line width=2pt] (4.5,3.5)-- (6.5,3.5);
\draw [line width=2pt] (6.5,3.5)-- (6.5,1);
\draw [line width=2pt] (6.5,1)-- (9,1);
\draw [line width=2pt] (2.5,3.5)-- (2.5,4.5);
\draw [line width=2pt,color=green] (0,5.021649433509434)-- (2,5.018701037033917);
\draw [line width=2pt,color=green] (4,2.5179763978882157)-- (6,2.5196501560453104);
\draw [line width=2pt,color=green] (6,3.51899027748033)-- (8,3.5225962109656295);
\draw [line width=2pt,color=green] (2,3.52066901913057)-- (4,3.517913512910509);

\draw [fill=green] (0,5.021649433509434) circle (5pt);
\draw [color=green] (2,5.018701037033917) circle (5pt);
\draw [fill=green] (2,3.52066901913057) circle (5pt);
\draw [fill=green] (4,2.5179763978882157) circle (5pt);
\draw [color=green] (6,2.5196501560453104) circle (5pt);
\draw [fill=green] (6,3.51899027748033) circle (5pt);
\draw [color=green] (8,3.5225962109656295) circle (5pt);
\draw [color=green] (4,3.517913512910509) circle (5pt);
\draw [fill=green] (8,1) circle (5pt);

\end{tikzpicture}}
    \caption{An Euler Characteristic Curve (black) and its vectorized version (green) with resolution $\Delta$. In this case, the vectorized version is stored as a vector of length 5 (the green filled-in points), but can be reconverted to a stepsize function.}
    \label{fig:ECC_vectorized}
\end{figure}

The vectorized ECC can be obtained by such vector as the union of $N-1$ left-closed, right-open intervals of length $\Delta$ that correspond to to sampling the value of the EC at filtration value $f_i$ and extending it till $f_{i+1}$. It makes sense then to ask whether it is possible to bound the difference between an ECC and its vectorized representation. Figure \ref{fig:ECC_vectorized} is an example of such difference when a curve is sampled in 5 points.

\begin{proposition}
\label{prop:ECC_vec}
Let $K$ be a filtered cell complex whose filtration values ranges from $0$ to $f_{max}$. The $L_1$ norm between the Euler Characteristic Curve of $K$ and its vectorized version at resolution $\Delta$ is bounded by
\begin{equation}
    || ECC(K) - vec(ECC(K), N)||_1 \leq \Delta( |K|/2 + F )
\label{eq:ECC_vect}
\end{equation}

where $|K|$ is the number of simplices in the complex and $F = \sum_{i=0}^{n-2}|EC(i \Delta) - EC((i+1)\Delta)|$ is the sum of the absolute value of the differences between consecutive values in the vectorized Euler Characteristic $vec(ECC(K), N)$.
\end{proposition}
\begin{proof}
We will prove the two terms in the bound separately as they come from two different types of errors. 

Type I errors occur when the EC  at two consecutive sampling points $f_i$ and $f_{i+1}$ is different. The simplest case is depicted in Figure \ref{fig:ECC_vect_proof}A, the EC changes values in between the sampling interval. We can upper bound this error with the area of the rectangle having as base the vectorization's resolution $\Delta=f_{i+1} - f_i$, and as height the difference between the EC at the two sampling points $|EC(f_i) - EC(f_{i+1})$. Note that this bound also holds in the more general case where the EC varies monotonically at multiple values inside the sampling interval. By summing up all the contributions we obtain the value $\Delta \cdot F = \Delta \cdot \sum_{i=0}^{n-2}|EC(i \Delta) - EC((i+1)\Delta)|$.

Type II errors, see Figure \ref{fig:ECC_vect_proof}B, occur when the EC has the same value at consecutive filtration steps but varies in between. The maximum possible variation can be upper bounded by the area of the rectangle with $\Delta$ as base and the half the number of cells in the complex as height. Each cell contributes to the EC by $\pm 1$, the factor one half is due to the constrain that the EC has the same value in $f_i$ and $f_{i+1}$. This amounts to the values $\Delta \cdot |K|/2$. 

By summing up the two contributions we obtain the bound in \ref{eq:ECC_vect}. Note that a generic situation can always be described as a combination of type I and type II errors.

\end{proof}

\begin{figure}[!ht]
    \centering
    \begin{subfigure}[t]{0.25\textwidth}
    \centering
        \scalebox{0.5}{\definecolor{ttzzqq}{rgb}{0.2,0.6,0}

\begin{tikzpicture}[line cap=round,line join=round,>=triangle 45,x=1cm,y=1cm]

\draw [line width=2pt] (0,0)-- (0,4);
\draw [line width=2pt] (4,0)-- (4,4);
\draw [line width=2pt] (0,3)-- (1,3);
\draw [line width=2pt] (1,1)-- (1,3);
\draw [line width=2pt] (1,1)-- (4,1);
\draw [line width=2pt,color=ttzzqq] (0,3.0368351897754455)-- (4,3.0368351897754455);

\draw[color=black] (0,-0.4) node {$f_i$};
\draw[color=black] (4,-0.4) node {$f_{i+1}$};
\draw [fill=ttzzqq] (0,3.0368351897754455) circle (5pt);
\draw [color=ttzzqq] (4,3.0368351897754455) circle (5pt);

\end{tikzpicture}}
    \caption{Type I}
    \end{subfigure}
    \begin{subfigure}[t]{0.25\textwidth}
    \centering
        \scalebox{0.5}{\definecolor{ttzzqq}{rgb}{0.2,0.6,0}

\begin{tikzpicture}[line cap=round,line join=round,>=triangle 45,x=1cm,y=1cm]

\draw [line width=2pt] (0,0)-- (0,4);
\draw [line width=2pt] (4,0)-- (4,4);
\draw [line width=2pt] (0,1.5)-- (1.5,1.5);
\draw [line width=2pt] (1.5,2.5)-- (1.5,1.5);
\draw [line width=2pt] (1.5,2.5)-- (2.5,2.5);
\draw [line width=2pt] (2.5,2.5)-- (2.5,1.5);
\draw [line width=2pt] (2.5,1.5)-- (4,1.5);
\draw [line width=2pt] (0,1.5)-- (1.5,1.5);
\draw [line width=2pt] (1.5,2.5)-- (2.5,2.5);
\draw [line width=2pt] (2.5,1.5)-- (4,1.5);
\draw [line width=2pt,color=ttzzqq] (0,1.5215607688849118)-- (4,1.5161281839226048);

\draw[color=black] (0,-0.4) node {$f_i$};
\draw[color=black] (4,-0.4) node {$f_{i+1}$};
\draw [fill=ttzzqq] (0,1.5215607688849118) circle (5pt);
\draw [color=ttzzqq] (4,1.5161281839226048) circle (5pt);

\end{tikzpicture}}
    \caption{Type II}
    \end{subfigure}
    \caption{The two possible source of errors during vectorization of an ECC}
    \label{fig:ECC_vect_proof}
\end{figure}
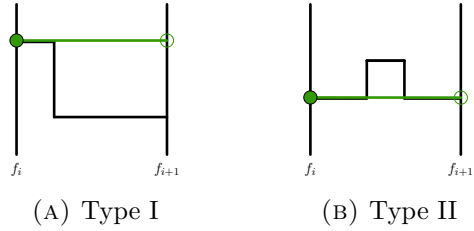

We have shown a way to bound the distance between an ECC and its vectorized version. Another possible stability question is whether this vectorization preserves distances between ECCs. In other words, we are interest in knowing whether something can be said for $||vec(ECC_1, N) - vec(ECC_2, N)||$ given $||ECC_1 - ECC_2||$  Unfortunately it is possible to construct examples in which two curves can be made arbitrary far apart but they have the same vectorization or two curves can be made arbitrary close but they have drastically different vectorizations. Figure \ref{fig:ECC_vs_vect_cases} shows two of such examples. Moreover, in the existing literature Johnson and Jung \cite{johnson_instability_2021} prove that the distance between two vectorized Betti curves can not be bounded by the Wasserstein distance between the respective persistence diagrams. They propose a stable vectorization inspired by Gaussian smoothing techniques. 

\begin{figure}[!ht]
    \centering
    \begin{subfigure}[t]{0.45\textwidth}
    \centering
        \scalebox{0.5}{\definecolor{ccqqqq}{rgb}{0.8,0,0}
\definecolor{ttzzqq}{rgb}{0.2,0.6,0}

\begin{tikzpicture}[line cap=round,line join=round,>=triangle 45,x=1cm,y=1cm]

\draw [->,line width=2pt] (0,0) -- (0,6);
\draw [->,line width=2pt] (0,0) -- (9,0);
\draw [line width=2pt,dash pattern=on 4pt off 4pt,color=ttzzqq] (2,0)-- (2,6);
\draw [line width=2pt,dash pattern=on 4pt off 4pt,color=ttzzqq] (4,0)-- (4,6);
\draw [line width=2pt,dash pattern=on 4pt off 4pt,color=ttzzqq] (6,0)-- (6,6);
\draw [line width=2pt,dash pattern=on 4pt off 4pt,color=ttzzqq] (8,0)-- (8,6);
\draw [line width=2pt] (0,5)-- (1.5,5);
\draw [line width=2pt] (1.5,3.5)-- (1.5,5);
\draw [line width=2pt] (1.5,3.5)-- (2.5,3.5);
\draw [line width=2pt] (2.5,2)-- (2.5,3.5);
\draw [line width=2pt] (2.5,2)-- (5.5,2);
\draw [line width=2pt] (5.5,1)-- (5.5,2);
\draw [line width=2pt] (5.5,1)-- (9,1);
\draw [line width=2pt,color=ccqqqq] (2.5,3.5)-- (2.5,5.5);
\draw [line width=2pt,color=ccqqqq] (2.5,5.5)-- (3.5,5.5);
\draw [line width=2pt,color=ccqqqq] (3.5,5.5)-- (3.5,2);
\end{tikzpicture}}
    \caption{}
    \end{subfigure}
    \begin{subfigure}[t]{0.45\textwidth}
    \centering
        \scalebox{0.5}{\definecolor{ccqqqq}{rgb}{0.8,0,0}
\definecolor{ttzzqq}{rgb}{0.2,0.6,0}

\begin{tikzpicture}[line cap=round,line join=round,>=triangle 45,x=1cm,y=1cm]
\draw [->,line width=2pt] (0,0) -- (0,6);
\draw [->,line width=2pt] (0,0) -- (9,0);
\draw [line width=2pt,dash pattern=on 4pt off 4pt,color=ttzzqq] (2,0)-- (2,6);
\draw [line width=2pt,dash pattern=on 4pt off 4pt,color=ttzzqq] (4,0)-- (4,6);
\draw [line width=2pt,dash pattern=on 4pt off 4pt,color=ttzzqq] (6,0)-- (6,6);
\draw [line width=2pt,dash pattern=on 4pt off 4pt,color=ttzzqq] (8,0)-- (8,6);
\draw [line width=2pt] (0,5)-- (3.5,5);
\draw [line width=2pt] (3.5,5)-- (3.5,1);
\draw [line width=2pt] (3.5,1)-- (9,1);
\draw [line width=2pt,color=ccqqqq] (3.5,5)-- (4.5,5);
\draw [line width=2pt,color=ccqqqq] (4.5,5)-- (4.5,1);
\draw [line width=2pt] (0,5)-- (3.5,5);
\end{tikzpicture}}
    \caption{}
    \end{subfigure}
    \caption{Two ECC superimposed in the same plot. In panel (A) the two curves can be made arbitrary far apart in $L_1$ but they have the same vectorization. In panel (B) the two curves can be made arbitrary close but they have drastically different vectorizations.}
    \label{fig:ECC_vs_vect_cases}
\end{figure}
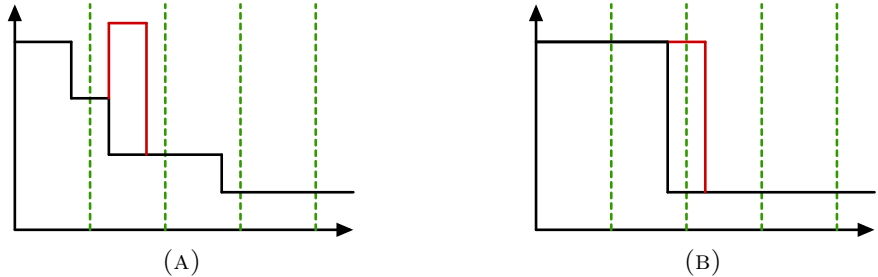

\subsection{Profiles}
An $n$-dimensional Euler Characteristic Profile whose filtration values ranges from $0$ to $f^i_{max}$ for $i\in 1 \cdots n$ can be vectorized in a similar fashion by sampling it on a grid of size $N_1 \times N_2 \times \cdots \times N_n$. In general the $N_i$ can be different and thus leading to different resolutions $\Delta_i$ on the various filtration parameters. The output of this sampling procedure is a $n$-dimensional tensor $vec(ECP, {N_i})$ that can be eventually flattened to a $1$-dimensional vector. 
Although this is an intuitive generalization of the $1$-dimensional ECC case, the procedure has an increased computational cost due to the difficulties in sampling EC values from a profile, as already discussed in Section \ref{sec:retrieving_EC}. Moreover, the stability result in \ref{prop:ECC_vec} can not be generalized to the multiparameter setting. As depicted in Figure \ref{fig:ECP_vectorized}, the grid vectorization could be not able to detect the contributions coming from pairs of cells. In the multiparameter case however, it is not possible to bound this contributions using only the vectorization resolutions $\Delta$ as such contribution can persist on subsequent grid elements up to infinity. 

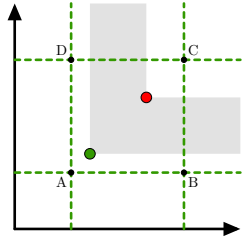
\begin{figure}[!ht]
    \centering
        \scalebox{0.5}{\definecolor{aqaqaq}{rgb}{0.6274509803921569,0.6274509803921569,0.6274509803921569}
\definecolor{ffqqqq}{rgb}{1,0,0}
\definecolor{ttzzqq}{rgb}{0.2,0.6,0}

\begin{tikzpicture}[line cap=round,line join=round,>=triangle 45,x=1cm,y=1cm]

\fill[line width=2pt,color=aqaqaq,fill=aqaqaq,fill opacity=0.3] (2,6) -- (2,2) -- (6,2) -- (6,3.5) -- (3.5,3.5) -- (3.5,6) -- cycle;
\draw [->,line width=2pt] (0,0) -- (0,6);
\draw [->,line width=2pt] (0,0) -- (6,0);
\draw [line width=2pt,dash pattern=on 4pt off 4pt,color=ttzzqq] (1.5,0)-- (1.5,6);
\draw [line width=2pt,dash pattern=on 4pt off 4pt,color=ttzzqq] (4.5,0)-- (4.5,6);
\draw [line width=2pt,dash pattern=on 4pt off 4pt,color=ttzzqq] (0,4.5)-- (6,4.5);
\draw [line width=2pt,dash pattern=on 4pt off 4pt,color=ttzzqq] (0,1.5)-- (6,1.5);

\draw [fill=ttzzqq] (2,2) circle (4pt);
\draw [fill=ffqqqq] (3.5,3.5) circle (4pt);
\draw [fill=black] (1.5,4.5) circle (2pt);
\draw[color=black] (1.25,4.75) node {D};
\draw [fill=black] (1.5,1.5) circle (2pt);
\draw[color=black] (1.25,1.25) node {A};
\draw [fill=black] (4.5,1.5) circle (2pt);
\draw[color=black] (4.75,1.25) node {B};
\draw [fill=black] (4.5,4.5) circle (2pt);
\draw[color=black] (4.75,4.75) node {C};

\end{tikzpicture}}
    \caption{A $2$-dimensional analog of a type II error of Figure \ref{fig:ECC_vect_proof}. The ECP is vectorized by sampling the EC values on the green grid. We can add pair of cells with contributions $\pm 1$ inside the rectangle ABCD in such a way that the value of the EC on the vertices does not change. However, such contributions have a non-zero sum on an area that can be made arbitrary large.}
    \label{fig:ECP_vectorized}
\end{figure}

\section{Examples and Experiments}

\subsection{RGB images}
\label{RGB}
A toy experiment using 3 dimensional Euler Characteristic Profiles can be constructed using RGB images. In a RGB image each pixel contains a tuple of 3 integers, each ranging from $0$ to $255$. They stand for the Red, Green and Blue color channel and all colors in the visible spectrum can be represented by a 3 tuple. In particular black is coded by (0,0,0) and white is (255, 255, 255). 

In this example we consider two different textures, stripes and checks, each of them can be red, green or blue. We generate 10 samples of each combination of style and color by adding random Gaussian noise to each pixel. We then compute the 3 dimensional Euler Characteristic Profile of the cubical complex obtained from each image and computed the matrix of pairwise $L_1$ distances between them. Such matrix is show in Figure \ref{fig:RGB_distance_matrix}. It confirms that distance between Euler Characteristic Profiles of different images increase following the intuitive sequence 'same style, same color' $<$ 'same style, different color' $<$ 'different style, same color' $<$ different style, different color'.

\begin{figure}[!ht]
    \centering
    \begin{subfigure}[c]{0.45\textwidth}
    \centering
        \includegraphics[width=0.8\textwidth]{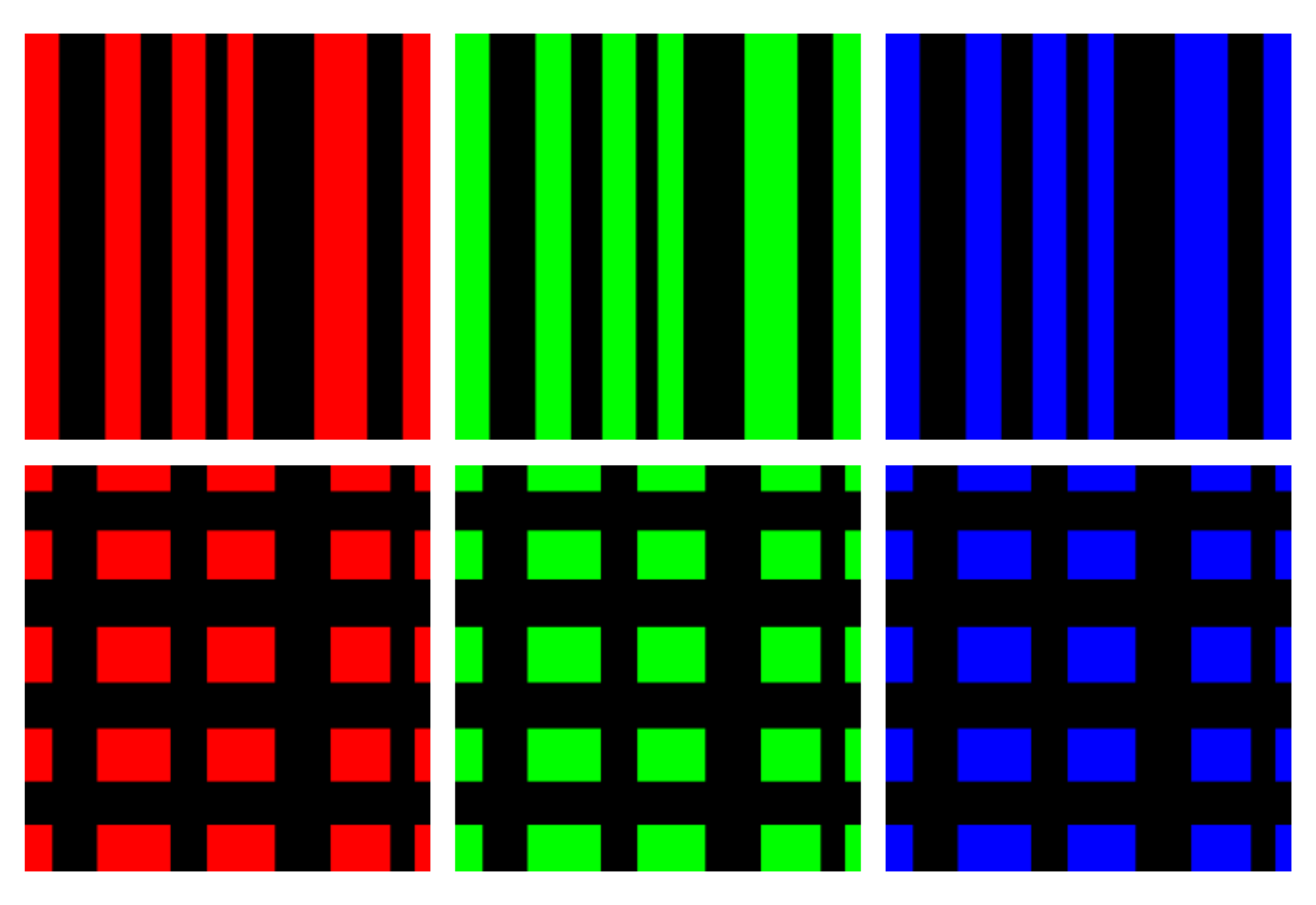}
    \end{subfigure}
    \begin{subfigure}[c]{0.45\textwidth}
    \centering
        \includegraphics[width=\textwidth]{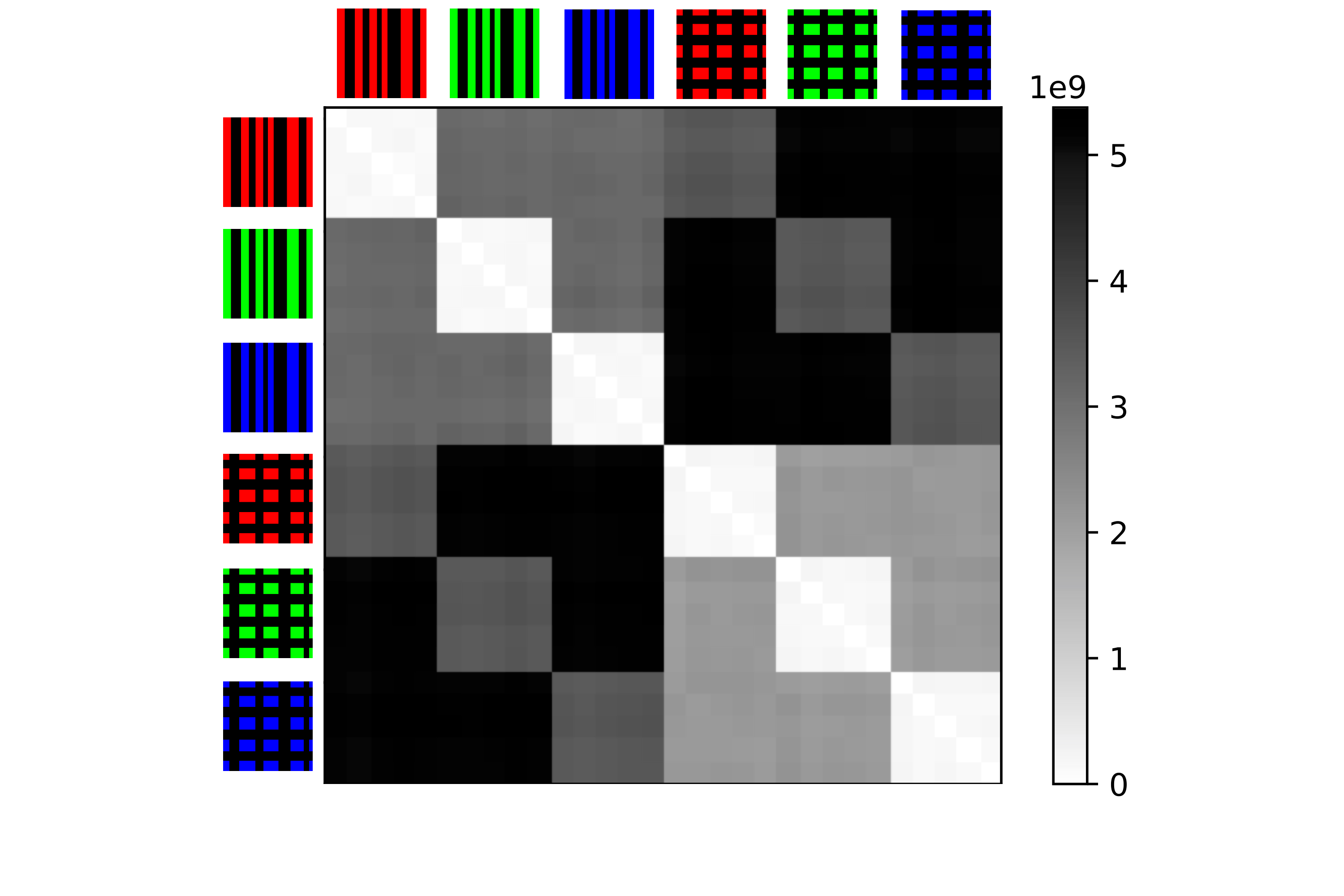}
    \end{subfigure}
    \caption{$60 \times 60$ distance matrix between Euler Characteristic Profiles of different RGB images.}
    \label{fig:RGB_distance_matrix}
\end{figure}

\subsection{Immune cell spatial patterns in tumors}
Vipond et al. \cite{vipond_multiparameter_2021} applied multiparameter persistent homology (MPH) landscapes to study immune cell location in digital histology images from head and neck cancer. They extracted the locations of three immune  cell types from histology slides thus obtaining a list of pointclouds labelled CD8+, FoxP3+, or CD68. The goal is to correctly classify a pointcloud. All pointcloud data are available at \href{https://github.com/MultiparameterTDAHistology/SpatialPatterningOfImmuneCells}{\texttt{github.com/MultiparameterTDAHistology/SpatialPatterningOfImmuneCells}}.
The Authors created a bifiltered Vietoris-Rips complex from each pointcloud, using radius and a codensity function defined over each vertex $p$ as  $\rho_{10}(p) = \frac{1}{10}\sum_{i=1}^{10} ||p - p_i||_2$ where $p_i$ is the $i$-th nearest neighbor of $p$. They then computed MPH-landscapes and used them as input for one of three classifiers: Linear Discriminant Analysis (LDA), Regularised Linear Discriminant Analysis (rLDA), and regularised Quadratic Discriminant Analysis (rQDA) \cite{hastie_elements_2009}. They made a randomized 80/20 training/test split, and evaluate the classification accuracy of 3 classifiers on the test data for each pair of cell types and for the three-class problem. The classification results are reported in the supplementary material of \cite{vipond_multiparameter_2021}.

We used the Authors' code to re-generate the same standard Vietoris-Rips and bifiltered Vietoris-Rips complexes from the provided pointclouds. We then computed ECC (radius only) and ECP (radius and codensity) for each complex and used them as input for the same LDA, rLDA and rQDA classifiers using the same train-test split procedure. The average accuracy for the various classification task  are reported in Tables \ref{table:accuracy_LDA}, \ref{table:accuracy_rLDA} and \ref{table:accuracy_rQDA}. Both ECC and ECP significatively outperforms MPH-landscapes while there is apparently no gain in moving from ECC to ECP. This can be an indication that the second dimension in the filtration (the codensity parameter) does not contain significant information.

\begin{table}[!ht]
    \begin{adjustwidth}{-1cm}{-1cm}
        \centering
        \begin{tabular}{|c|c|c|c|c|}
        \hline
            & CD68+ vs FoxP3+ & CD8+ vs FoxP3+ & CD8+ vs CD68+ & CD8+ vs CD68+ vs FoxP3+ \\ \hline
            & MPL - ECC - ECP & MPL - ECC - ECP  & MPL - ECC - ECP  & MPL - ECC - ECP \\ \hline
        T\_A & 0.584 - 0.938 - \textbf{0.941} & 0.672 - \textbf{0.994} - 0.988 & 0.669 - \textbf{0.894} - 0.856 & 0.486 - \textbf{0.896} - 0.886 \\ \hline
        T\_B & 0.794 - 0.917 - \textbf{0.922} & 0.88 - \textbf{0.992} - \textbf{0.992} & 0.54 - 0.943 - \textbf{0.962} & 0.568 - 0.921 - \textbf{0.940} \\ \hline
        T\_C & 0.723 - \textbf{0.947} - 0.904 & 0.7 - \textbf{0.884} - 0.859 & 0.605 - \textbf{0.811} - 0.699 & 0.505 - \textbf{0.842} - 0.755 \\ \hline
        T\_D & 0.811 - \textbf{0.960} - 0.933 & 0.899 - \textbf{0.986} - 0.985 & 0.644 - 0.802 - \textbf{0.807} & 0.613 - 0.862 - \textbf{0.874} \\ \hline
        T\_E & 0.732 - \textbf{0.941} - 0.940 & 0.644 - 0.867 - \textbf{0.869} & 0.593 - \textbf{0.806} - 0.688 & 0.511 - \textbf{0.842} - 0.719 \\ \hline
        T\_F & 0.738 - 0.655 - \textbf{0.933} & 0.644 - 0.619 - \textbf{0.830} & 0.73 - 0.709 - \textbf{0.850} & 0.511 - 0.578 - \textbf{0.824} \\ \hline
        T\_G & 0.771 - 0.788 - \textbf{0.858} & 0.782 - 0.791 - \textbf{0.904} & \textbf{0.675} - 0.614 - 0.609 & 0.599 - \textbf{0.673} - 0.659 \\ \hline
        T\_H & 0.710 - 0.651 - \textbf{0.885} & 0.682 - 0.747 - \textbf{0.955} & 0.628 - 0.695 - \textbf{0.891} & 0.555 - 0.659 - \textbf{0.845} \\ \hline
        T\_I & 0.733 - \textbf{0.788} - 0.737 & \textbf{0.758} - 0.716 - 0.679 & 0.540 - 0.693 - \textbf{0.713} & 0.548 - \textbf{0.716} - 0.493 \\ \hline
        T\_J & 0.727 - 0.642 - \textbf{0.767} & 0.535 - 0.678 - \textbf{0.857} & 0.602 - 0.808 - \textbf{0.868} & 0.449 - 0.507 - \textbf{0.699} \\ \hline
        T\_K & 0.510 - \textbf{0.872} - 0.770 & 0.570 - 0.784 - \textbf{0.816} & 0.502 - 0.823 - \textbf{0.877} & 0.404 - 0.594 - \textbf{0.635} \\ \hline
        T\_N & 0.493 - 0.457 - \textbf{0.570} & 0.512 - \textbf{0.658} - 0.632 & 0.577 - 0.507 - \textbf{0.760} & 0.342 - \textbf{0.462} - 0.370 \\ \hline
        T\_O & \textbf{0.948} - 0.830 - 0.840 & \textbf{0.788} - 0.602 - 0.754 & 0.532 - 0.484 - \textbf{0.598} & 0.550 - 0.431 - \textbf{0.615} \\ \hline
    \end{tabular}
\end{adjustwidth}
\caption{Average classification accuracy for the LDA classifier using as input MLP, ECC or ECP. Data for each tumor are split into 80/20 train-test splits and classification accuracy is reported as the mean over 100 repetitions of splitting, training and testing. }
\label{table:accuracy_LDA}
\end{table}

\begin{table}[!ht]
    \begin{adjustwidth}{-1cm}{-1cm}
        \centering
        \begin{tabular}{|c|c|c|c|c|}
        \hline
            & CD68+ vs FoxP3+ & CD8+ vs FoxP3+ & CD8+ vs CD68+ & CD8+ vs CD68+ vs FoxP3+ \\ \hline
            & MPL - ECC - ECP & MPL - ECC - ECP  & MPL - ECC - ECP  & MPL - ECC - ECP \\ \hline
        T\_A & 0.491 - \textbf{0.967} - 0.964 & 0.642 - \textbf{0.973} - 0.967 & 0.630 - \textbf{0.840} - 0.830 & 0.427 - 0.858 - \textbf{0.859} \\ \hline
        T\_B & 0.760 - \textbf{0.892} - 0.869 & 0.787 - \textbf{0.986} - 0.985 & 0.671 - 0.942 - \textbf{0.945} & 0.604 - \textbf{0.868} - 0.865 \\ \hline
        T\_C & 0.863 - \textbf{0.906} - 0.896 & 0.747 - \textbf{0.847} - 0.842 & \textbf{0.653} - 0.584 - 0.614 & \textbf{0.640} - 0.628 - 0.627 \\ \hline
        T\_D & 0.683 - \textbf{0.926} - 0.918 & 0.829 - \textbf{0.990} - 0.988 & 0.476 - \textbf{0.779} - \textbf{0.779} & 0.492 - \textbf{0.779} - 0.775 \\ \hline
        T\_E & 0.820 - \textbf{0.886} - 0.883 & 0.736 - \textbf{0.929} - 0.920 & 0.534 - 0.735 - \textbf{0.743} & 0.502 - \textbf{0.702} - 0.683 \\ \hline
        T\_F & 0.623 - 0.899 - \textbf{0.925} & 0.476 - 0.842 - \textbf{0.847} & 0.765 - 0.909 - \textbf{0.921} & 0.408 - 0.845 - \textbf{0.847} \\ \hline
        T\_G & 0.886 - \textbf{0.932} - 0.927 & 0.897 - 0.970 - \textbf{0.975} & 0.446 - \textbf{0.696} - 0.692 & 0.581 - 0.738 - \textbf{0.746} \\ \hline
        T\_H & 0.524 - 0.890 - \textbf{0.898} & 0.735 - \textbf{0.930} - 0.929 & 0.714 - \textbf{0.882} - 0.877 & 0.502 - 0.844 - \textbf{0.859} \\ \hline
        T\_I & 0.859 - 0.912 - \textbf{0.931} & 0.883 - 0.908 - \textbf{0.909} & 0.484 - 0.470 - \textbf{0.474} & 0.597 - \textbf{0.619} - 0.614 \\ \hline
        T\_J & 0.608 - \textbf{0.763} - 0.750 & 0.750 - 0.835 - \textbf{0.872} & 0.850 - 0.882 - \textbf{0.892} & 0.536 - 0.653 - \textbf{0.670} \\ \hline
        T\_K & 0.376 - \textbf{0.868} - 0.804 & 0.523 - \textbf{0.918} - 0.914 & 0.455 - \textbf{0.857} - 0.845 & 0.261 - \textbf{0.718} - 0.679 \\ \hline
        T\_N & 0.410 - 0.527 - \textbf{0.563} & 0.432 - 0.662 - \textbf{0.745} & 0.643 - 0.690 - \textbf{0.713} & 0.294 - 0.388 - \textbf{0.460} \\ \hline
        T\_O & 0.702 - \textbf{0.954} - 0.952 & 0.644 - \textbf{0.806} - 0.772 & 0.546 - 0.672 - \textbf{0.684} & 0.429 - \textbf{0.639} - 0.632 \\ \hline
    \end{tabular}
\end{adjustwidth}
\caption{Average classification accuracy for the rLDA classifier using as input MLP, ECC or ECP. Data for each tumor are split into 80/20 train-test splits and classification accuracy is reported as the mean over 100 repetitions of splitting, training and testing. }
\label{table:accuracy_rLDA}
\end{table}

\begin{table}[!ht]
    \begin{adjustwidth}{-1cm}{-1cm}
    \centering
    \begin{tabular}{|c|c|c|c|c|}
        \hline
            & CD68+ vs FoxP3+ & CD8+ vs FoxP3+ & CD8+ vs CD68+ & CD8+ vs CD68+ vs FoxP3+ \\ \hline
            & MPL - ECC - ECP & MPL - ECC - ECP  & MPL - ECC - ECP  & MPL - ECC - ECP \\ \hline
            T\_A & 0.503 - \textbf{0.945} - 0.931 & 0.598 - \textbf{0.840} - 0.838 & 0.598 - \textbf{0.840} - 0.838 & 0.380 - \textbf{0.865} - 0.861 \\ \hline
            T\_B & 0.738 - \textbf{0.896} - 0.867 & 0.588 - \textbf{0.913} - 0.911 & 0.588 - \textbf{0.913} - 0.911 & 0.531 - \textbf{0.871} - 0.869 \\ \hline
            T\_C & 0.855 - \textbf{0.915} - 0.906 & 0.673 - 0.552 - \textbf{0.568} & 0.673 - 0.552 - \textbf{0.568} & 0.614 - 0.627 - \textbf{0.640} \\ \hline
            T\_D & 0.554 - \textbf{0.934} - 0.929 & 0.494 - \textbf{0.767} - 0.755 & 0.494 - \textbf{0.767} - 0.755 & 0.482 - 0.786 - \textbf{0.787} \\ \hline
            T\_E & 0.826 - \textbf{0.876} - 0.871 & 0.548 - 0.751 - \textbf{0.754} & 0.548 - 0.751 - \textbf{0.754} & 0.499 - \textbf{0.754} - 0.724 \\ \hline
            T\_F & 0.646 - \textbf{0.964} - 0.963 & 0.666 - 0.853 - \textbf{0.855} & 0.666 - 0.853 - \textbf{0.855} & 0.412 - \textbf{0.881} - 0.878 \\ \hline
            T\_G & 0.882 - \textbf{0.937} - 0.928 & 0.485 - \textbf{0.723} - 0.699 & 0.485 - \textbf{0.723} - 0.699 & 0.583 - \textbf{0.771} - 0.767 \\ \hline
            T\_H & 0.621 - 0.968 - \textbf{0.967} & 0.699 - 0.886 - \textbf{0.898} & 0.699 - 0.886 - \textbf{0.898} & 0.550 - 0.889 - \textbf{0.901} \\ \hline
            T\_I & 0.919 - 0.928 - \textbf{0.940} & 0.493 - \textbf{0.531} - 0.527 & 0.493 - \textbf{0.531} - 0.527 & \textbf{0.626} - 0.621 - 0.624 \\ \hline
            T\_J & 0.588 - \textbf{0.908} - 0.903 & 0.860 - 0.898 - \textbf{0.902} & 0.860 - 0.898 - \textbf{0.902} & 0.541 - \textbf{0.720} - 0.719 \\ \hline
            T\_K & 0.468 - \textbf{0.874} - 0.838 & 0.567 - \textbf{0.923} - 0.903 & 0.567 - \textbf{0.923} - 0.903 & 0.352 - \textbf{0.751} - 0.736 \\ \hline
            T\_N & 0.353 - 0.453 - \textbf{0.477} & 0.510 - \textbf{0.617} - 0.600 & 0.510 - \textbf{0.617} - 0.600 & 0.334 - \textbf{0.392} - 0.384 \\ \hline
            T\_O & 0.724 - 0.972 - \textbf{0.984} & 0.524 - \textbf{0.668} - 0.662 & 0.524 - \textbf{0.668} - 0.662 & 0.440 - 0.730 - \textbf{0.738} \\ \hline
    \end{tabular}
    \end{adjustwidth}
    \caption{Average classification accuracy for the rQDA classifier using as input MLP, ECC or ECP. Data for each tumor are split into 80/20 train-test splits and classification accuracy is reported as the mean over 100 repetitions of splitting, training and testing. }
\label{table:accuracy_rQDA}
\end{table}

\subsection{Prostate cancer histology slides}
Lawson et al. \cite{lawson_persistent_2019} demonstrated that Persistent Homology can successfully be used to evaluate features in prostate cancer hematoxylin and eosin (H\&E) stained slides. Their dataset, available in the Open Science Framework \url{https://osf.io/k96qw/} contains $5182$ RGB images of a resolution $512 \times 512$ corresponding to different regions of interest (ROIs) in prostate cancer H\&E slices obtained from 39 patients. Each image is labelled with a Gleason score of 3, 4 or 5 indicating the architectural patterns of the cancer. An higher Gleason score indicates an increasing level of cancer aggressiveness. The datasets contains $2567$ grade 3 ROIs, $2351$ grade 4 ROIs but only $264$ grade 5 ROIs. Given the unbalance in the data we decided to consider a classification problem between grade 3 and 4.

Following the procedure described by the Authors we normalized and extracted the hematoxylin and eosin color channel from each ROI. By doing so we converted each RGB image into a bidimensional (H, E) one. We first computed the ECC for each of the grayscale images corresponding to the hematoxylin channel as it is the color that highlights cell nuclei. We then also used the eosin color channel to obtain a $2$-dimensional ECP. We input either the ECCs or the ECPs into an Support Vector Machine (SVM) \cite{bishop2006pattern} classifier and computed the mean test accuracy over 100 rounds with a 80/20 training split. The results are displayed in Table \ref{tab:HE_acc}. The classifier using as input the $2$-dimensional ECPs is consistently performing better than the one using the $1$-dimensional ECCs.

\begin{table}[!ht]
    \centering
    \begin{tabular}{c|c}
    hematoxylin ECC & hematoxylin \& eosin ECP \\
    \hline
      $0.765 \pm 0.001$       & $0.826 \pm 0.001$
    \end{tabular}
    \caption{Mean test accuracy for the Gleason 3 vs Gleason 4 classification using ECCs or ECPs as input to an SVM classifier.}
    \label{tab:HE_acc}
\end{table}

\begin{figure}[!ht]
\begin{subfigure}[t]{0.3\textwidth}
    \centering
    \includegraphics[width=0.7\textwidth]{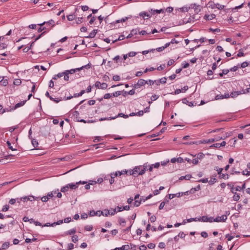}
\subcaption{}
\end{subfigure}
\hfill
\begin{subfigure}[t]{0.3\textwidth}
    \centering
    \includegraphics[width=0.7\textwidth]{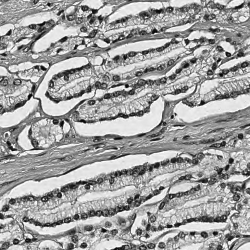}
\subcaption{}
\end{subfigure}
\hfill
\begin{subfigure}[t]{0.3\textwidth}
    \centering
    \includegraphics[width=0.7\textwidth]{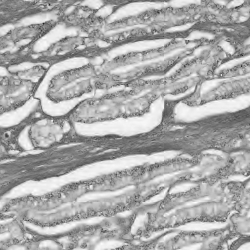}
\subcaption{}
\end{subfigure}
\caption{Panel (A) depicts a raw RGB ROI. Panel (B) contains the hematoxylin channel while panel (C) contains the eosin one.}
\label{fig:prostate_he}
\end{figure}

\begin{figure}[!ht]
\centering
\begin{subfigure}[t]{0.3\textwidth}
    \centering
    \includegraphics[height=0.9\textwidth]{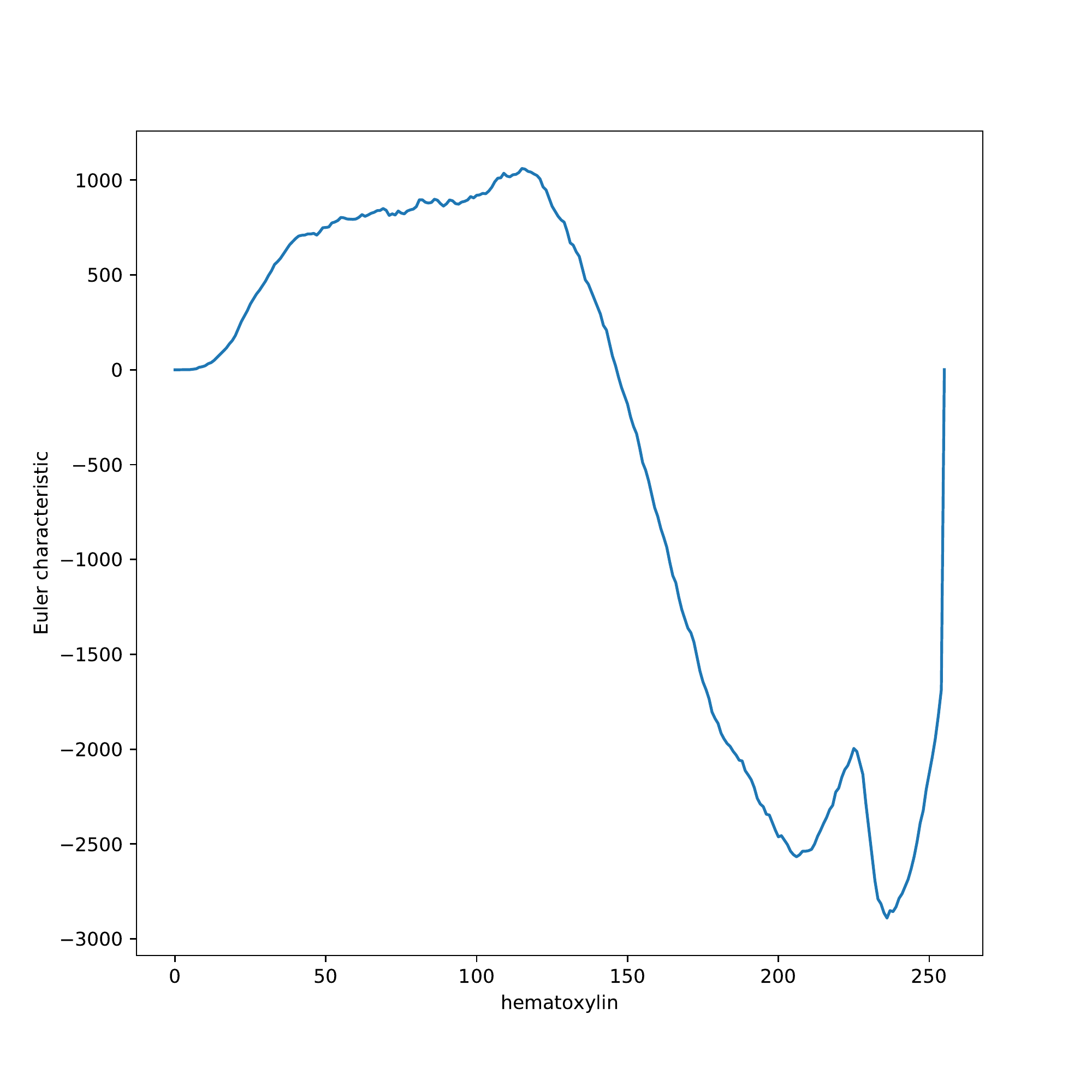}
\subcaption{}
\end{subfigure}
\hfill
\begin{subfigure}[t]{0.3\textwidth}
    \centering
    \includegraphics[height=0.9\textwidth]{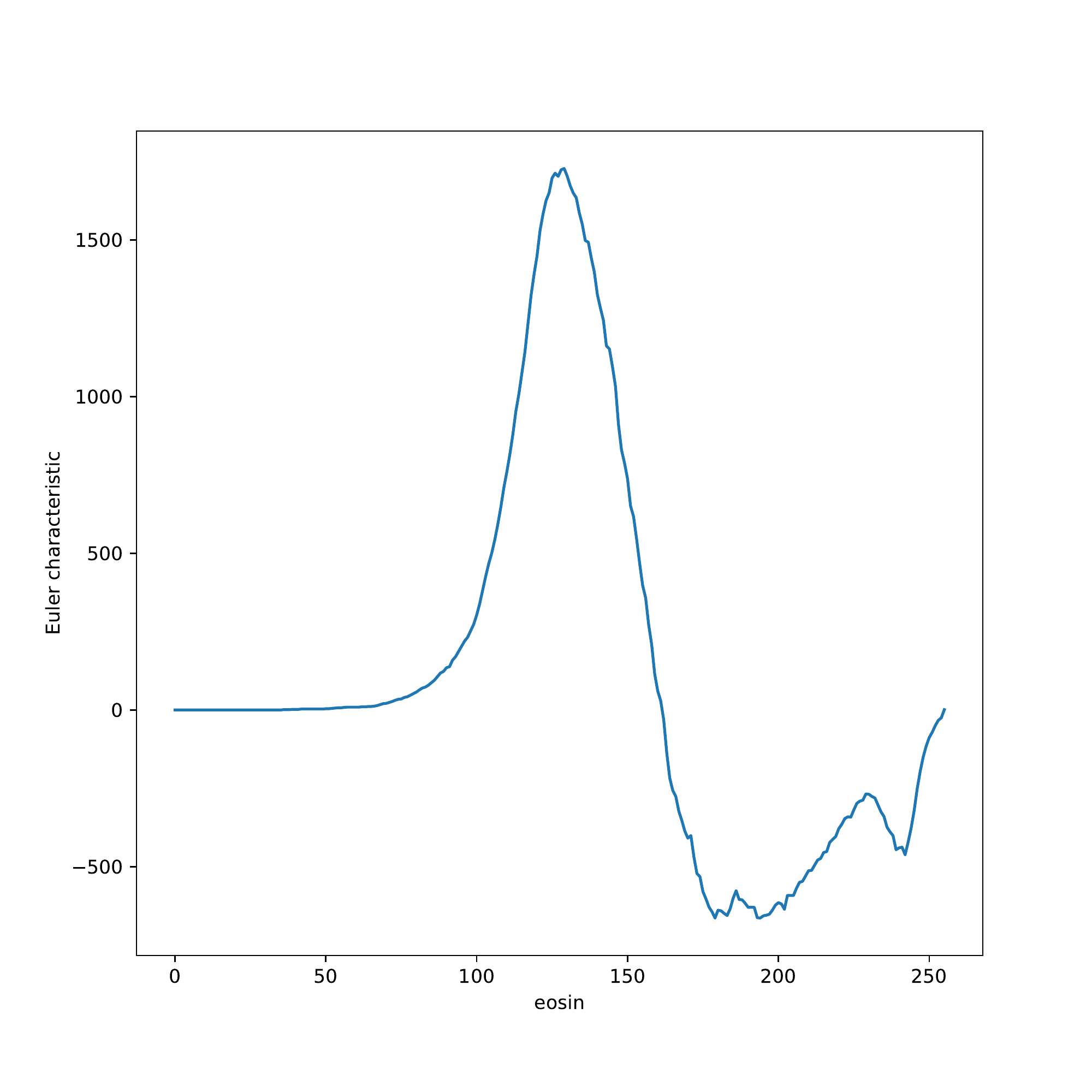}
\subcaption{}
\end{subfigure}
\hfill
\begin{subfigure}[t]{0.3\textwidth}
    \centering
    \includegraphics[height=0.9\textwidth]{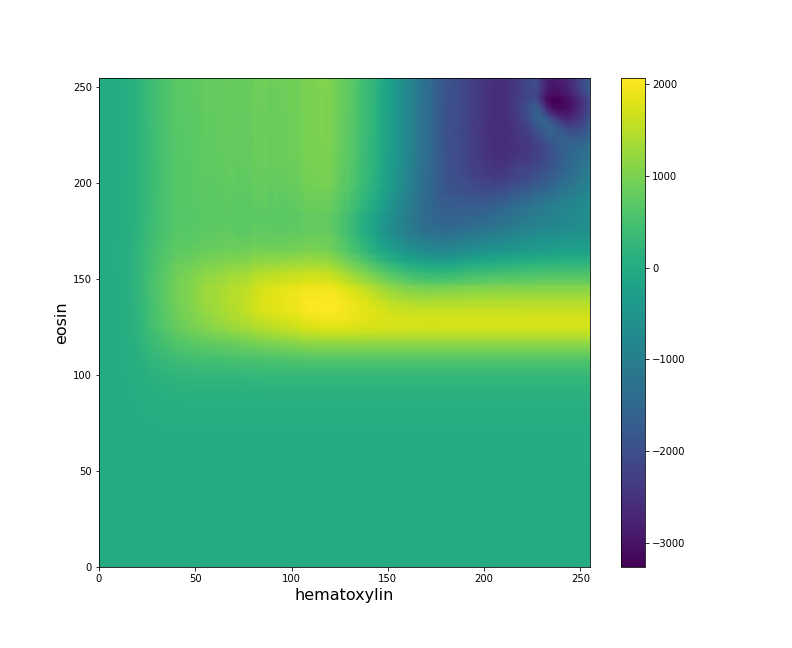}
\subcaption{}
\end{subfigure}
\caption{Panel (A) depicts the hematoxylin ECC for the ROI in Figure \ref{fig:prostate_he} while panel (B) depicts the eosin ECC. The combined ECP is showed in panel (C).}
\label{fig:prostate_ECC_ECP}
\end{figure}

\section{Conclusions}
Euler Characteristic Curves and Profiles provide a stable summary of the shape of data. Unlike other summaries used in Topological Data Analysis this one can be computed in a distributed fashion, hence is applicable to deal with big data problems. In addition we show, contrary to a common misconception, that the Euler Characteristic Curves and Profiles enjoys certain type of stability. We confirm it when using them to discriminate various toy datasets with varying level of noise. We also show how to compare and vectorize the Euler Characteristic Curves and Profiles and apply them to a number of real data analysis problems. The presented results are accompanied with efficient Python implementation. For example, on modern commodity hardware, our implementation for V-R complexes can handle a number of simplices on the order of $10^{10}$. This is two order of magnitude more that what can be achieved using available software like GUDHI \cite{gudhi:urm}.
With this work we hope that the machinery of Euler Characteristic Curves and Profiles will be useful for practitioners in Topological Data Analysis.

\section{Code availability}
Python implementations of Algorithms \ref{algo:local_contribution_VR} and \ref{algo:local_contribution_cube} are available at \url{https://github.com/dioscuri-tda/pyEulerCurves}. Jupyter notebooks to reproduce all the experiments described in this article are available at \url{https://github.com/dioscuri-tda/ecp_experiments}.

\newpage
\appendix
\section{Time performance analysis}
\label{apdx:time}
We asses the time performance of Algorithm \ref{algo:local_contribution_VR} by analyzing the worst-case scenario, a complete graph built from a pointcloud $\{x_i\}$ $i \in [1, n]$. This is the worst-case scenario as it contains the maximal number of cliques (hence simplices) for a given number of vertices, namely $2^n -1$.  As discussed in the previous section, the running time will be dominated by the first vertex $x_1$ as it has the highest number of successive neighbours. 
The most time consuming operations are the ones that happen inside Algorithm \ref{algo:increase_dim}, namely the \textsc{update filtration} and \textsc{update common neighbours} subroutines.

\subsection{Update filtration}
The extension of a $d$-clique requires checking whether one or more of the new introduced edges have a filtration value higher than the current $d$-clique. Comparison between floats can be done in constant time and has to be repeated $d$ times. With reference to figure \ref{fig:extend_filtration}, we can assign to each edge in the simplex tree a cost that depends only on the edge depth. The total sum of such cost is
\[
\sum_{i=1}^n (i-1) \binom{n}{i} = 2^{n - 1} n - 2^n + 1 \quad .
\]
In case of perfect parallelization, the cost for the first vertex only is
\[
\sum_{i=1}^{n - 1} i \binom{n - 1}{i} = 2^{n - 2} (n - 1) \quad .
\]

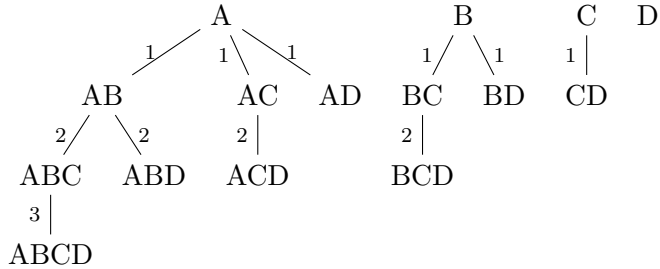
\begin{figure}[!ht]
\centering
      \begin{forest}
[
  [A, no edge
    [AB, edge label={node[midway,left,font=\scriptsize]{1}}
      [ABC, edge label={node[midway,left,font=\scriptsize]{2}}
        [ABCD, edge label={node[midway,left,font=\scriptsize]{3}}]
      ]
      [ABD, edge label={node[midway,right,font=\scriptsize]{2}}]
    ]
    [AC, edge label={node[midway,left,font=\scriptsize]{1}}
      [ACD, edge label={node[midway,left,font=\scriptsize]{2}}]
    ]
    [AD, edge label={node[midway,right,font=\scriptsize]{1}}]
  ]
  [B, no edge
    [BC, edge label={node[midway,left,font=\scriptsize]{1}}
      [BCD, edge label={node[midway,left,font=\scriptsize]{2}}]
    ]
    [BD, edge label={node[midway,right,font=\scriptsize]{1}}]
  ]
  [C, no edge
    [CD, edge label={node[midway,left,font=\scriptsize]{1}}]
  ]
  [D, no edge]
]
\end{forest} 
\caption{Simplex tree for a 4-clique with the update filtration cost.}
\label{fig:extend_filtration}
\end{figure}

\subsection{Update common neighbours}
Updating the list of common neighbours after a clique extension requires computing the intersection between the current list of common neighbours (with length $m_1$) and the list of neighbours of the newly added vertex (with length $m_2$). Given that such lists are ordered their intersection can be computed in $\mathcal{O}(m_1 + m_2)$.  The total cost for this operation can be obtain recursively by observing that the number of neighbours in a clique is uniquely determined - in this particular case - only by the last element of the clique. For example, in Figure \ref{fig:intersection} the subtree spanning from \textsc{AB} is the same as the one spanning from \textsc{B}, and the one from \textsc{AC} is equivalent to \textsc{C}.   The total cost for a clique of size $n$ can be then expressed as twice the cost for the $(n-1)-$clique plus the cost of the depth $1$ edges spanning from the first vertex:
\begin{align*}
c(n) &= \frac{(n-1)(n-2)}{2} + 2c(n-1) + (n-1)^2 \quad ;\\
c(2) &= 1 \quad .
\end{align*}
Such recurrence has solution 
\[
c(n) = \frac{1}{2} (-8 + 2^{n+3} - 5 n - 3 n^2) \quad .
\]
In case of perfect parallelization, the cost for the first vertex only is
\[
c(n) - c(n-1) = 2^{n+1} - 3 n - 1 \quad .
\]

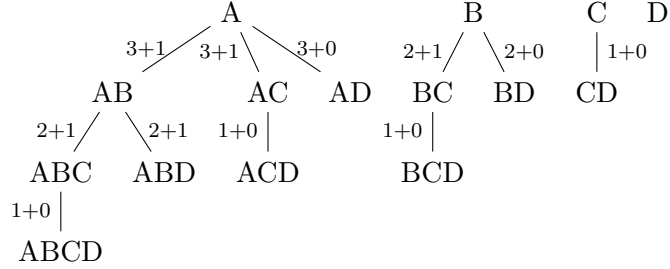
\begin{figure}[!ht]
\centering
      \begin{forest}
[
  [A, no edge
    [AB, edge label={node[midway,left,font=\scriptsize]{3+1}}
      [ABC, edge label={node[midway,left,font=\scriptsize]{2+1}}
        [ABCD, edge label={node[midway,left,font=\scriptsize]{1+0}}]
      ]
      [ABD, edge label={node[midway,right,font=\scriptsize]{2+1}}]
    ]
    [AC, edge label={node[midway,left,font=\scriptsize]{3+1}}
      [ACD, edge label={node[midway,left,font=\scriptsize]{1+0}}]
    ]
    [AD, edge label={node[midway,right,font=\scriptsize]{3+0}}]
  ]
  [B, no edge
    [BC, edge label={node[midway,left,font=\scriptsize]{2+1}}
      [BCD, edge label={node[midway,left,font=\scriptsize]{1+0}}]
    ]
    [BD, edge label={node[midway,right,font=\scriptsize]{2+0}}]
  ]
  [C, no edge
    [CD, edge label={node[midway,right,font=\scriptsize]{1+0}}]
  ]
  [D, no edge]
]
\end{forest} 
\caption{Simplex tree for a 4-clique with the update common neighbours cost.}
\label{fig:intersection}
\end{figure}

\section{Memory performance analysis}
\label{apdx:memory}
At each step, the algorithm needs to store in memory the local graph $G$ with each edge's filtration value. The local graph can be stored as an adjacency matrix whose entries represent the filtration values. Moreover, we need to store the current list of simplices, the list of their filtration values and the list of common neighbour for each simplex. Let us denote with $V$ the bits needed to store an edge label (usually a \textsc{uint}) and with $F$ the bits needed to store a filtration value (usually a \textsc{float}).
Assuming the worst case scenario of a fully connected graph with $n$ nodes, the maximum number of simplices will be generated at dimension $\frac{n}{2}$ and will be $\binom{n}{n/2}$. The memory cost at that step will then be $\mathcal{O}(n(n-1)F + \binom{n}{n/2}nV + \binom{n}{n/2}F) = O \left( \frac{2^n}{\sqrt{n}} (nV + F) \right)$, where the first term is the graph cost, the second one is the cost of the list of simplices and the list of common neighbours that we assume to have the same size due to the symmetry of the binomial coefficients, and the third one is the cost of the simplices filtration values.

\bibliographystyle{plain}
\bibliography{references}

\end{document}